\documentclass[12pt]{amsart}

\usepackage[usenames]{color}
\usepackage{amssymb}
\usepackage{amsmath}
\usepackage{amsthm}
\usepackage{amsfonts}
\usepackage{enumerate}
\usepackage{url}
\usepackage{comment}

\theoremstyle{plain}
\newtheorem{theorem}{Theorem}[section]
\newtheorem{corollary}[theorem]{Corollary}
\newtheorem{lemma}[theorem]{Lemma}
\newtheorem{proposition}[theorem]{Proposition}

\theoremstyle{definition}
\newtheorem{definition}[theorem]{Definition}
\newtheorem{example}[theorem]{Example}
\newtheorem{conjecture}[theorem]{Conjecture}
\newtheorem{remark}[theorem]{Remark}

\makeatletter
\@namedef{subjclassname@2020}{%
  \textup{2020} Mathematics Subject Classification}
\makeatother

\newcommand{\ordp}{\nu_p}
\newcommand{\ordfp}{\overline{\nu}_{\mathfrak{p}}}
\newcommand{\ordfpp}{\nu_{\mathfrak{p}}}

\newcommand{\ordfPP}{\nu_{\mathfrak{P}}}

\newcommand{\incp}{\mathrm{inc}_p}
\newcommand{\incfp}{\mathrm{inc}_{\mathfrak{p}}}

\newcommand{\dKp}{D_{K,\mathfrak{p}}}

\newcommand{\pfloor}[1]{\lfloor #1 \rfloor_{p}}
\newcommand{\ld}{\mathrm{ld}}
\newcommand{\ldQ}{\mathrm{ld}_{\mathbb{Q},p}}
\newcommand{\ldQp}{\mathrm{ld}_{\mathbb{Q}_p,p}}
\newcommand{\ldKp}{\mathrm{ld}_{K,\mathfrak{p}}}

\begin{document}

\baselineskip=17pt

\title{A Generalization of Arithmetic Derivative to $p$-adic Fields and Number Fields}

\author{Brad Emmons}
\address{Utica University \\ 1600 Burrstone Road \\ Utica, NY 13502 \\ USA}
\email{bemmons@utica.edu}

\author{Xiao Xiao}
\address{Utica University \\ 1600 Burrstone Road \\ Utica, NY 13502 \\ USA}
\email{xixiao@utica.edu}

\date{}

\begin{abstract}
The arithmetic derivative is a function from the natural numbers to itself that sends all prime numbers to $1$ and satisfies the Leibniz rule. The arithmetic partial derivative with respect to a prime $p$ is the $p$-th component of the arithmetic derivative. In this paper, we generalize the arithmetic partial derivative to $p$-adic fields (the local case) and the arithmetic derivative to number fields (the global case). We study the dynamical system of the $p$-adic valuation of the iterations of the arithmetic partial derivatives. We also prove that for every integer $n\geq 0$, there are infinitely many elements with exactly $n$ anti-partial derivatives. In the end, we study the $p$-adic continuity of arithmetic derivatives.
\end{abstract}

\subjclass[2020]{Primary: 11A25, Secondary: 11R04}

\keywords{Arithmetic derivative, arithmetic partial derivative, $p$-adic fields, number fields, $p$-adic continuity}

\maketitle 

\section{Introduction}

Let $\mathbb{N} = \{0, 1, 2, \ldots\}$. The arithmetic derivative is a function $D: \mathbb{N} \to \mathbb{N}$ that satisfies the following two properties: $D(p) = 1$ for all primes $p$, and the Leibniz rule, $D(xy) = D(x)y+xD(y)$ for all $x, y \in \mathbb{N}$. One of the questions on the 1950 Putnam competition \cite{putnam} asked the contestants to predict the limit of the sequence $63, D(63), D^2(63), \ldots$. Many sources cite this as the origin of the arithmetic derivative.  However we were able to find a paper by Shelly \cite{shelly} published in 1911 which introduced this topic as well as some of the basic properties and generalizations of this function.

One can ask a more general question.  If we fix $x \in \mathbb{N}$, what is the limit of the sequence $x, D(x), D^2(x), \ldots$. This is not easy to predict in general. Ufnarovski and \AA hlander made the following conjecture.

\begin{conjecture}{\cite[Conjecture 2]{Ufnarovski}} \label{conjecture:barbeau}
For every $x \in \mathbb{N}$, exactly one of the following could happen: either $D^i(x) = 0$ or $p^p$ for some prime $p$ for sufficiently large $i$, or $\displaystyle \lim_{i \to +\infty} D^i(x) = +\infty$.
\end{conjecture}

We note that Shelly \cite{shelly} alluded to this conjecture and Barbeau \cite{barbeau} made a similar conjecture. One corollary of this conjecture is that if the sequence $x, D(x), D^2(x), \ldots$ is eventually periodic, then the period is $1$.  That is $D^k(x) = p^p$ for some prime $p$ when $k \gg 0$. Given $y > 1$, it is not hard to show \cite[Corollary 3]{Ufnarovski} that there are finitely many (possibly 0) $x$ such that $D(x) = y$. We call $x$ an anti-derivative of $y$. Ufnarovski and \AA hlander made the following conjecture.

\begin{conjecture}{\cite[Conjecture 8]{Ufnarovski}} \label{conjecture:ufnarovski}
For every integer $n \geq 0$ there are infinitely many $x> 0$ such that $x$ has exactly $n$ anti-derivatives.
\end{conjecture}

Let $\ordp$ be the $p$-adic valuation. One can show that $D(0) = 0$ and for $x > 0$, $D$ has the following explicit formula
\[D(x) = x \sum_p \dfrac{\ordp(x)}{p}.\]
This is a finite sum as there are only finitely many $p$ such that $\ordp(x) \neq 0$. It is natural to generalize $D$ to $\mathbb{Q}$ as $\ordp$ is well-defined over $\mathbb{Q}$.  We will use $D$ to denote the arithmetic derivative defined on $\mathbb{Q}$ in the introduction section. This generalization allows positive integers to have more anti-derivatives than they have in $\mathbb{N}$. For example, $2$ does not have an anti-derivative in $\mathbb{N}$ but $D(-21/16) = 2$. The only anti-derivatives of $1$ in $\mathbb{N}$ are the prime numbers but $D(-5/4) = 1$. Another direction to generalize $D$ is, instead of differentiating with respect to all prime numbers, we only differentiate with respect to a set of primes. More specifically, let $T \subset \mathbb{P}$ be a nonempty set of rational primes. For $0 \neq x \in \mathbb{Q}$, we define 
\[D_{\mathbb{Q}, T}(x) = x \sum_{p \in T} \dfrac{\ordp(x)}{p}.\]
This is called the arithmetic subderivative over $\mathbb{Q}$ with respect to $T$, first introduced by Haukkanen, Merikoski, and Tossavainen \cite{HMT4}. If $T = \mathbb{P}$, then $D_{\mathbb{Q}, T} = D$. If $T = \{p\}$ contains a single prime number, then $D_{\mathbb{Q},T} = D_{\mathbb{Q},p}$ is called the arithmetic partial derivative with respect to $p$.

The authors of this paper have proved \cite{EX1} that the following sequence of integers
\[\ordp(x),\; \ordp(D_{\mathbb{Q},p}(x)),\; \ordp(D^2_{\mathbb{Q},p}(x)),\; \ldots\]
is eventually periodic of period $\leq p$. An immediate corollary of this result is a positive answer to a conjecture similar to Conjecture \ref{conjecture:barbeau} in the case of arithmetic partial derivative. We have to replace $p^p$ in Conjecture \ref{conjecture:barbeau} by $bp^p$ where $\ordp(b) =0$ since $D_{\mathbb{Q},p}(bp^p) = bp^p$. In the same paper, we also proved a criterion to determine when an integer has integral anti-partial derivatives, and as application, we gave a positive answer to Conjecture \ref{conjecture:ufnarovski} in the case of arithmetic partial derivative.

A natural next step is to generalize the arithmetic derivative to number fields and their rings of integers. The Leibniz rule can be used to generalize $D$ to all unique factorization domains (UFD) $R$. In every equivalence class $\{x \textrm{ irreducible in } R \mid x = ux', u \in R^{\times}\}$, we choose an element $x_0$ and define $D_R(x_0) = 1$ (similar to $D(p) = 1$). For all units $u \in R^{\times}$, we define $D_R(u) = 0$ (similar to $D(\pm 1) = 0$). By the unique factorization property and the Leibniz rule, we can extend the definition of $D$ to the entire ring $R$ as well as its field of fraction $\textrm{Frac}(R)$. Let $\mathcal{P}$ be a set of chosen irreducible elements as described above, one from each equivalence classes. For every $x \in \textrm{Frac}(R)$, if $x = up_1\cdots p_k q_1^{-1}\cdots q_{\ell}^{-1}$ with $u \in R^{\times}$ and $p_i, q_j \in \mathcal{P}$ ($p_i, q_j$ are not necessarily pairwise different) then 
\[D_R(x) = x\Big(\sum_{i=1}^k \dfrac{1}{p_i} - \sum_{j=1}^{\ell} \dfrac{1}{q_j}\Big).\]
There are two major obstacles with this generalization. First, for every number field $K$, it is well known that $\mathcal{O}_K$ is not necessarily a UFD. It has been proved that this idea will fail for non-UFD \cite{Haukkanen}. Second, this definition of $D(x)$ depends on the choice of irreducible elements set $\mathcal{P}$ as well as the ring. There is no canonical way to choose $x_0$ within each equivalence classes. Also, for an irreducible element $x \in \mathcal{P} \subset R$, we have $D_R(x) = 1$. But if we consider $x \in \textrm{Frac}(R)$ and define $D$ over $\textrm{Frac}(R)$, then we will get $D_{\textrm{Frac}(R)}(x) = 0$ since all nonzero elements of $\textrm{Frac}(R)$ are invertible. In other words, suppose $x \in R_1 \subset R_2$, we do not necessarily have $D_{R_1}(x) = D_{R_2}(x)$. This phenomenon makes it hard to generalize $D$ to all number fields in a consistent way using this definition.

To get around the first obstacle, Mistri and Pandey \cite{Pandey} defined the arithmetic derivative of an ideal in the ring of integers $\mathcal{O}_K$ of a number field $K$. This generalization uses the fact that every fractional ideal of $K$ can be uniquely factorized into a product of prime ideals in $\mathcal{O}_K$. Suppose $I = \mathfrak{p}_1 \mathfrak{p}_2 \cdots \mathfrak{p}_k$ is an ideal of $\mathcal{O}_K$ where $\mathfrak{p}_i$ are primes ideals of $\mathcal{O}_K$ with $\mathfrak{p}_i \mid p_i$ (again $\mathfrak{p}_i$ and $p_i$ are not necessarily pairwise different). Then the arithmetic derivative of $I$ is an ideal of $\mathcal{O}_K$ defined by
\[D_K(I) = \Big( p_1p_2\cdots p_k \sum_{i=1}^k \dfrac{1}{p_i} \Big).\]
This means that the arithmetic derivative of every ideal of $\mathcal{O}_K$ is a principal ideal in $\mathcal{O}_K$ generated by an integer. From the definition, it is easy to see that $D_{\mathbb{Z}}(n) = (D(n))$ where $D_{\mathbb{Z}}(n)$ is the arithmetic derivative of the ideal $(n)$ and $D(n)$ is the usual arithmetic derivative of an integer. This coincidence is certainly nice as part of the generalization but the second obstacle mentioned above still exists. For example, let $K = \mathbb{Q}(i)$ and we have $2\mathcal{O}_K = (1+i)(1-i)$, hence $D_K(2\mathcal{O}_K)= 4\mathcal{O}_K$. On the other hand, $D_{\mathbb{Z}}(2\mathbb{Z}) = \mathbb{Z}$. This means that if $x \in K_1 \subset K_2$, we do not necessarily have $D_{K_1}(x\mathcal{O}_{K_1}) \subset D_{K_2}(x\mathcal{O}_{K_2})$.

In this paper, we propose a new way to define the arithmetic derivative (resp. the arithmetic subderivative) $D_K$ (resp. $D_{K,T}$) on every finite Galois extension $K/\mathbb{Q}$ in a consistent way in the following sense. First $D_K(x) = D(x)$ for all $x \in \mathbb{Q}$, so $D_K$ is a true extension of $D$ from $\mathbb{Q}$ to $K$. Second, if $K_1$ and $K_2$ are two finite Galois extensions, then for every $x \in K_1 \cap K_2$, we have $D_{K_1}(x) = D_{K_2}(x)$. This means that the definition of arithmetic derivative of $x$ does not depend on the choice of the Galois extension. Because the arithmetic derivative satisfies $D_K(x)/x \in \mathbb{Q}$, we can even generalize it to every number field $L/\mathbb{Q}$ (not necessarily Galois) by taking a restriction $D_L(x) := D_{K}(x) = x \cdot (D_{K}(x)/x) \in L$ where $K$ is a finite Galois extension containing $x$. Please refer to Section \ref{section:numberfields} for detailed definition.

At the local level, suppose $K$ is a finite extension of the $p$-adic rational numbers $\mathbb{Q}_p$.  Let $\ordfpp$ be the unique valuation on $K$ that extends the $p$-adic valuation $\ordp$ on $\mathbb{Q}$. It only makes sense to study the arithmetic partial derivative $\dKp$ over $K$. As part of the study of the behavior of the sequence $x, \dKp(x), \dKp^2(x), \ldots$, we give a complete description of the behavior of the following so-called $\ordfpp$ sequence of $x$
\[\ordfpp(x),\; \ordfpp(\dKp(x)),\; \ordfpp(\dKp^2(x)),\; \ldots.\]

\begin{theorem} \label{theorem:main0}
Let $K$ be a finite extension over $\mathbb{Q}_p$ and $\mathfrak{p}$ be the unique prime ideal of $\mathcal{O}_K$. For every $x \in K$, we have the following three properties.
\begin{enumerate}
    \item If $\ordfpp(\ordfpp(x)) \geq 0$ or $\ordfpp(x) \in \{ 0, +\infty \} $, then the $\ordfpp$ sequence of $x$ is eventually periodic of period $\leq p$. 
    \item If $\ordfpp(\ordfpp(x)) < 0$, then the $\ordfpp$ sequence of $x$ converges to $-\infty$.
    \item The $\ordfpp$ sequence of $x$ is eventually $+\infty$ if and only if \[\ordfpp(x) \in \{0, 1, \ldots, p-1, +\infty\}.\]
    
\end{enumerate}
\end{theorem}

See Lemma \ref{lemma:eventuallyinfinity}, Proposition \ref{prop:vp-infty}, and Theorem \ref{theorem:main1} for a proof of Theorem \ref{theorem:main0}. Using the same idea as in our previous paper \cite{EX1}, we are also able to give a positive answer to Conjecture \ref{conjecture:ufnarovski} in the $p$-adic fields case as well.

\begin{theorem}[Theorem \ref{theorem:infinite_derivative}] \label{theorem:main2}
Let $K$ be a finite extension over $\mathbb{Q}_p$. For each positive integer $n$, there are infinitely many $x_0 \in K$ such that $\dKp(x_0)$ has exactly $n$ anti-partial derivatives in $K$.
\end{theorem}

One difficulty of studying the iteration of arithmetic derivatives is that the arithmetic derivative is neither additive nor a group homomorphism. But if one considers the so-called logarithmic derivative $\ld(x) := D(x)/x$, it is not hard to see that $\ld: \mathbb{Q}^{\times} \to \mathbb{Q}$ is a group homomorphism from the multiplicative group to the additive group, just like the usual logarithmic function. As we generalize $D$ to $D_K$, we also study the generalization of $\ld$ to $\ld_K$. In particular, we have shown that $\ld_K(K^{\times})$ are also isomorphic as subgroups of $\mathbb{Q}$ for any finite Galois extension $K$; see Theorem \ref{theorem:galiso}. We also give a concrete description of the exact image of $\ld_K(K^{\times})$ when $K$ is a quadratic extension.

It is not surprising that the arithmetic derivative function $D$ is not continuous over $\mathbb{Q}$ because given two rational numbers that are close by (under the Archimedean metric), their prime factorizations can be drastically different. In fact, Haukkanen, Merikoski and Tossavainen \cite{HMT5} have shown that for every $x \in \mathbb{Q}$, the arithmetic subderivative $D_{\mathbb{Q},T}$ (and in particular the arithmetic derivative) can obtain arbitrary large values in any small neighbourhood of $x$. Therefore $D_{\mathbb{Q},T}$ is clearly not continuous with respect to the standard Archimedean topology of $\mathbb{Q}$. But what about the $p$-adic topology? In another paper, Haukkanen, Merikoski and Tossavainen \cite{HMT3} have proved that the arithmetic partial derivative $D_{\mathbb{Q},p}$ is always continuous. They have also shown in some cases, the arithmetic subderivative $D_{\mathbb{Q},T}$ can be continuous at some points but discontinuous at other points. Major cases have been left open, for example, it is unknown whether $D_{\mathbb{Q},T}$ is continuous or not at nonzero points when $T$ is an infinite set. As we generalize arithmetic partial derivatives to $p$-adic local fields and arithmetic subderivative to number fields, it makes sense to study whether the generalizations are $\mathfrak{p}$-adically continuous or not. We state our results in two theorems, one for the local case and one for the global case.

\begin{theorem} \label{theorem:mainlocal}
Suppose $K$ is a number field. Let $\mathfrak{p}$ be a prime ideal of $\mathcal{O}_K$. Then the arithmetic partial derivative $\dKp$ is $\mathfrak{p}$-adically continuous at every point in $K$. Moreover $\dKp$ is strictly differentiable and twice strictly differentiable (with respect to the ultrametic $|\cdot|_{\ordfpp}$) at every nonzero point in $K$ but $\dKp$ is not strictly differentiable (with respect to the ultrametic $|\cdot|_{\ordfpp}$) at $0$.
\end{theorem}

See Theorems \ref{theorem:cont_partial}, \ref{theorem:local_diff}, and \ref{theorem:local_not_diff_0} for a proof of Theorem \ref{theorem:mainlocal}. The same result is true for arithmetic partial derivative over $p$-adic fields.

\begin{theorem} \label{theorem:mainglobal}
Suppose $K$ is a number field. Let $\mathfrak{p}$ be a prime ideal and $T$ be a nonempty subset of prime ideals of $\mathcal{O}_K$.
\begin{enumerate}
    \item The arithmetic subderivative $D_{K,T}$ is $\mathfrak{p}$-adically continuous but not strictly differentiable (with respect to the ultrametic $|\cdot|_{\ordfpp}$) at $0$.
    \item If $T \neq \{\mathfrak{p}\}$, then the arithmetic subderivative $D_{K,T}$ is $\mathfrak{p}$-adically discontinuous at every nonzero point in $K$.
\end{enumerate}
\end{theorem}

See Theorems \ref{theorem:cont_0}, \ref{theorem:not_diff}, \ref{theorem:discontinuous}, and \ref{theorem:discontinuous_special} for a proof of Theorem \ref{theorem:mainglobal}. By letting $K = \mathbb{Q}$ and $\mathfrak{p} = (p)$ in Theorem \ref{theorem:mainglobal}, we are able to give answers to all the open questions in \cite[Section 7]{HMT3}.

\section{$p$-adic Fields} \label{section:p-adic}

\subsection{Definition}

Fix a rational prime $p$. Let $\mathbb{Q}_p$ be the field of $p$-adic rational numbers and $\ordp$ the $p$-adic valuation. We denote the $p$-adic absolute value by $|\cdot|_{\ordp}$. Recall that the arithmetic partial derivative (with respect to $p$) $D_{\mathbb{Q},p} : \mathbb{Q} \to \mathbb{Q}$ is defined by
\begin{displaymath}
    D_{\mathbb{Q},p}(x) := \begin{cases}
        x\ordp(x)/p, & \text{if $x \neq 0$};\\
        0, & \text{if $x=0$.}
        \end{cases}
\end{displaymath}
One can extend $D_{\mathbb{Q},p}$ to $D_{\mathbb{Q}_p,p}$ with the same formula because $\ordp$ is well-defined on $\mathbb{Q}_p$. We can further extend $D_{\mathbb{Q}_p,p}$ to $p$-adic fields because $\ordp$ can be uniquely extended to a discrete valuation over $p$-adic fields. Let $K$ be a finite extension of $\mathbb{Q}_p$ of degree $n = [K:\mathbb{Q}_p]$. Let $\mathcal{O}_K$ be the ring of integers, which is a discrete valuation ring with maximal ideal $\mathfrak{p}$ and residue field $\mathcal{O}_K/\mathfrak{p}$. Let $f = f(K|\mathbb{Q}_p) = [\mathcal{O}_K/\mathfrak{p}:\mathbb{F}_p]$ be the inertia degree and $e = e(K|\mathbb{Q}_p)$ the ramification index, that is, the unique integer such that $p\mathcal{O}_K = \mathfrak{p}^e$. We have $n = e f$. It is well known \cite[Chapter 2 Proposition 3]{localfields} that $K$ is again complete with respect to the $\mathfrak{p}$-adic topology. There exists a unique discrete valuation $\ordfpp: K \to \mathbb{Q} \cup \{+\infty\}$ that extends $\ordp$ defined by 
\[\ordfpp(x) := \dfrac{1}{n}\ordp(N_{K/\mathbb{Q}_p}(x)),\]
where $N_{K/\mathbb{Q}_p}: K \to \mathbb{Q}_p$ is the norm. We know that $\ordfpp(K) = \mathbb{Z}/e$. For every $x \in K$, we set $k = k(x) :=\ordfpp(\ordfpp(x))$, so $k \geq -\ordp(e)$. The discrete valuation $\ordfpp$ defines a unique absolute value on $K$, which will be denoted by $|\cdot|_{\ordfpp}$, that extends the $p$-adic absolute value on $\mathbb{Q}_p$:
\[|x|_{\ordfpp} = \sqrt[\leftroot{-5}\uproot{10}n]{\Big|N_{K/\mathbb{Q}_p}(x)\Big|_{\ordp}}.\]
We can extend $D_{\mathbb{Q}_p,p}$ to $\dKp: K \to K$ as follows:
\begin{displaymath}
    \dKp(x) := \begin{cases}
        x\ordfpp(x)/p, & \text{if $x \neq 0$};\\
        0, & \text{if $x=0$.}
        \end{cases}
\end{displaymath}
One can check that $\dKp$ satisfies the Leibniz rule. It is evident that $\dKp(x) = D_{\mathbb{Q}_p,p}(x)$ for all $x \in \mathbb{Q}_p$. Note that the definition of $\dKp$ is independent of the choice of uniformizers of $\mathcal{O}_K$.

Let $K$ and $K'$ be two finite extensions over $\mathbb{Q}_p$ such that $x \in K \cap K' =: K''$. Let $\ordfpp$, $\nu_{\mathfrak{p}'}$, $\nu_{\mathfrak{p}''}$ be the unique discrete valuations that extend $\ordp$ to $K$, $K'$, and $K''$ respectively. Clearly $\ordfpp|_{K''} = \nu_{\mathfrak{p}'}|_{K''} = \nu_{\mathfrak{p}''}$. Therefore we have $\dKp(x) = x\ordfpp(x)/p = x\nu_{\mathfrak{p}''}(x)/p = x \nu_{\mathfrak{p}'}(x)/p = D_{K', \mathfrak{p}'}(x) \in K \cap K'$. This implies that the definition of arithmetic partial derivative of $x$ is independent of the choice of finite extensions where $x$ lies. 

\begin{remark}
Let $q$ be another prime different from $p$. The $q$-adic valuation $\nu_q$ defined on $\mathbb{Q}$ does not extend to $\mathbb{Q}_p$ or finite extensions of $\mathbb{Q}_p$. Therefore, unlike the case of $\mathbb{Q}$ where we have one arithmetic partial derivative for each prime number, there is only one well-defined arithmetic partial derivative for $\mathbb{Q}_p$ and for finite extensions of $\mathbb{Q}_p$.
\end{remark}

\subsection{Periodicity of $\ordfpp$ sequence}
Let $K/\mathbb{Q}_p$ be a finite extension and let $x \in K$. Let $\mathfrak{p}$ be the maximal ideal of $\mathcal{O}_K$ and $\ordfpp$ the unique discrete valuation that extends $\nu_p$. We call the following sequence
\[\ordfpp(x),\; \ordfpp(\dKp(x)),\; \ordfpp(\dKp^2(x)),\; \ldots\]
the $\ordfpp$ sequence of $x$. Note that the $\ordfpp$ sequence of $x$ is independent of the choice of $K$. If $\ordfpp(\dKp^j(x)) = +\infty$ for some integer $j \geq 0$, then $\dKp^j(x) = 0$ and thus $\dKp^i(x) = 0$ for all $i \geq j$. If $\ordfpp(\dKp^i(x)) < +\infty$ for all $i \geq 0$, then we call the sequence of increments of consecutive terms 
\[\ordfpp(\dKp(x)) - \ordfpp(x), \ordfpp(\dKp^2(x)) - \ordfpp(\dKp(x)), \ordfpp(\dKp^3(x)) - \ordfpp(\dKp^2(x)), \ldots\]
the $\incfp$ sequence of $x$. Suppose $\ordfpp(x) = bp^k$ where $\ordp(b) = 0$ and $k \geq -\ordp(e)$. Then the increment is
\begin{equation} \label{equation:ordp_inc}
\ordfpp(\dKp(x)) - \ordfpp(x) = \ordfpp(\frac{\ordfpp(x)}{p}) = \ordfpp(bp^{k-1}) = k-1 = \ordfpp(\ordfpp(x))-1. 
\end{equation}

\begin{lemma} \label{lemma:eventuallyinfinity}
The following two statements are equivalent:
\begin{enumerate}
    \item The $\ordfpp$ sequence of $x$ is eventually $+\infty$.
    \item $\ordfpp(x) \in \{0, 1, 2, \ldots, p-1, +\infty\}$.
\end{enumerate}
\end{lemma}
\begin{proof}
Suppose $\ordfpp(x) \in \{0, 1, 2, \ldots, p-1, + \infty\}.$  If $\ordfpp(x) = + \infty$, then $x = 0$, and $\dKp(x) =0$ for all $n \geq 0$.  If $\ordfpp(x) = 0$, then $x$ is a unit in $\mathcal{O}_K$, and thus $\dKp^n(x) = 0$ for all $n \geq 1$.  If $\ordfpp(x) = j$ for some $1 \leq j \leq p-1$, then $\ordfpp(\dKp^i(x)) = j -i$ for $1 \leq i \leq j$.  From $\ordfpp(\dKp^i(x)) = 0$ we get $\dKp^i(x)$ is a unit in $\mathcal{O}_K$, and thus $\dKp^n(x) = 0$ for all $n > j$.

Now we show that if $\ordfpp(x) \not \in \{ 0, 1, 2, \ldots, p-1, +\infty\}$, then the $\ordfpp$ sequence of $x$ is not eventually $+\infty$. It suffices to show that $\ordfpp(\dKp^i(x)) \neq 0$ for all $i \geq 0$. We consider three mutually disjoint cases.

\begin{itemize}
    \item[Case 1.] Suppose $\ordfpp(x) \not \in \mathbb{Z}$. Then $\ordfpp(\dKp(x)) \not \in  \mathbb{Z}$ by \eqref{equation:ordp_inc}. By induction, we get $\ordfpp(\dKp^i(x)) \not \in \mathbb{Z}$ since $\ordfpp(\ordfpp(\dKp^{i-1}(x)))-1 \in \mathbb{Z}$. In particular, $\ordfpp(\dKp^i(x)) \neq 0$.
    \item[Case 2.] Suppose $\ordfpp(x) \geq p$ is an integer. If $p  \nmid \ordfpp(x)$, then $\ordfpp(x) > p$ and $k = 0$, and so $\ordfpp( \dKp(x)) = \ordfpp(x) - 1 \geq p$. If $p \mid \ordfpp(x)$, then $k \geq 1$, and thus $\ordfpp( \dKp(x)) \geq \ordfpp(x) \geq p$ by \eqref{equation:ordp_inc}. Therefore $\ordfpp(\dKp(x)) \geq p > 0$. By induction, we get $\ordfpp(\dKp^i(x)) \neq 0$.
    \item[Case 3.] Suppose $\ordfpp(x) = bp^k < 0$ is an integer. Since $| bp^k | \geq p^k > k - 1,$ we get $\ordfpp( \dKp(x)) = bp^k + (k-1) < 0$. By induction, we get $\ordfpp(\dKp^i(x)) \neq 0$.
\end{itemize}
Combining all three cases, we have proved that if $\ordfpp(x) \notin \{ 0, 1, 2, \ldots, p-1, +\infty\}$, then the $\ordfpp$ sequence of $x$ is not eventually $+\infty$.
\end{proof}

\begin{remark}
Ufnarovski and {\AA}hlander conjecture \cite[Conjecture 8]{Ufnarovski} that there exists an infinite sequence $a_n$ of different natural numbers such that $a_1 = 1$ and $D_{\mathbb{Q}}(a_n) = a_{n-1}$ for $n \geq 2$. Here $D_{\mathbb{Q}}$ is the arithmetic derivative (not arithmetic partial derivative) defined on $\mathbb{Q}$. The same question can be asked for $\dKp$. Suppose there exists an infinite sequence $a_n \in K$ such that $a_1 = 1$ and $\dKp(a_n) = a_{n-1}$ for $n \geq 2$. Let $N = p+1$ and we know that the $\ordfpp$ sequence of $a_N$ is eventually $+\infty$ because $\ordfpp(\dKp^N(a_N)) = \ordfpp(\dKp(a_1)) = \ordfpp(0) = +\infty$. By the proof of Lemma \ref{lemma:eventuallyinfinity}, we know that $\ordfpp(a_2) = 1, \ordfpp(a_3) = 2, \ldots, \ordfpp(a_{N-1}) = p-1$, and there does not exist $a_N$ such that $\dKp(a_N) = a_{N-1}$. Hence the conjecture is false over $K$ for arithmetic partial derivative. On a related note, if we let $a_1 \in K \backslash \mathcal{O}_K^{\times}$ for some finite extension $K/\mathbb{Q}_p$, then it is possible to find an infinite sequence $a_n \in K$ such that $\dKp(a_n) = a_{n-1}$ for all $n \geq 2$. For example, let $K = \mathbb{Q}$, $a_1 = p^{p^2}$, and for all $m \geq 1$, let $a_{2m} = p^{p^2+1}/(p^2+1)^m$ and $a_{2m+1} = p^{p^2}/(p^2+1)^m$. It is easy to check that $D_{\mathbb{Q},p}(a_{2m+1}) = a_{2m}$ and $D_{\mathbb{Q},p}(a_{2m}) = a_{2m-1}$.  
\end{remark}

The next proposition tells us if $\ordfpp(\ordfpp(x)) < 0$, then the $\incp$ sequence of $x$ is constant and negative. As a result of that, the $\ordfpp$ sequence of $x$ converges to $-\infty$.

\begin{proposition} \label{prop:vp-infty}
Let $x \in K$ be a nonzero element such that $\ordfpp(x) = bp^{k}$ with $\ordp(b) = 0$ and $k < 0$.  Then the $\incfp$ sequence of $x$ is a constant sequence with negative terms
\[ (k-1, k-1, k-1, \ldots ). \]
As a result, the $\ordfpp$ sequence of $x$ converges to $-\infty$.
\end{proposition}
\begin{proof}
Equation \eqref{equation:ordp_inc} implies that the first term of the $\incfp$ sequence of $x$ is indeed $k-1$. Since 
\[\ordfpp(x) + (k-1) = bp^{k} +(k-1) = p^{k}(b + (k-1)p^{-k})\]
where $\ordp(b + (k-1)p^{-k}) = 0$, we can write $\ordfpp(\dKp(x)) = b'p^{k}$ where $b' := b+(k-1)p^{-k}$ with $\ordp(b') = 0$. Since $\ordfpp(\ordfpp(\dKp(x))) = \ordfpp(\ordfpp(x))$, we see that the second term of the $\incfp$ sequence of $x$ is again $k-1$. In the meantime, we can write  $\ordfpp(\dKp^2(x)) = b''p^{k}$ for some $b'' := b'+(k-1)p^{-k}$ where $\ordp(b'') = 0$. By induction, we see that every term of the  $\incfp$ sequence of $x$ is equal to $k-1$. Therefore $\ordfpp(\dKp^n(x)) = \ordfpp(x) + n(k-1) \rightarrow -\infty$ as $n \rightarrow \infty$.  
\end{proof}

If the $\ordfpp$ sequence of $x$ is eventually $+\infty$, then it is periodic of period $1$. For the rest of this subsection, we assume that the $\ordfpp$ sequence of $x$ is not eventually $+\infty$ and $\ordfpp(\ordfpp(x)) > 0$. We will show that under these conditions, the $\ordfpp$ sequence of $x$ is eventually periodic of period $\leq p$. The next proposition gives a recipe of the initial terms of the $\incfp$ sequence of $x$ if $\ordfpp(\ordfpp(x)) > 0$.

\begin{proposition} \label{prop:segment}
    Let $x \in K$ be a nonzero element such that $\ordfpp(x) = bp^{k}$ with $\ordp(b) = 0$ and $k > 0$.  Denote $k' := (k-1 \bmod p) + 1 \leq p$. The first $k'$ terms of the $\incfp$ sequence of $x$ are
    \[(k-1, \underbrace{-1, -1, \ldots, -1}_{(k-1 \bmod p) \text{ copies}}).\]
\end{proposition}
\begin{proof}
The first term of the $\incfp$ sequence of $x$ is indeed $k-1$ by \eqref{equation:ordp_inc}. We have
\[\ordfpp(\dKp(x)) = bp^{k} + (k-1).\]
If $k'=1$, then there is nothing further to prove. If $k'=2$, we have $k \equiv 2 \pmod p$ and thus $p \nmid (bp^k + (k-1))$. By \eqref{equation:ordp_inc} again, we get the second term of the $\incfp$ sequence of $x$ is
   \[\ordfpp(\dKp^2(x)) - \ordfpp(\dKp(x)) = - 1\]
and $\ordfpp(\dKp^2(x)) = bp^k + (k-2)$. The proof is complete by induction on $k'$.
\end{proof}

\begin{corollary} \label{corollary:basecase}
Let $x \in K$ be a nonzero element such that $\ordfpp(x) = bp^{k}$ with $\ordp(b) = 0$ and $1 \leq k \leq p$. Then the $\ordfpp$ sequence and the $\incfp$ sequence of $x$ are periodic of period $k$. 
\end{corollary}

\begin{proof}
If $1 \leq k \leq p$, then $k' = (k-1 \bmod p) + 1 = k-1+1= k$. The first $k+1$ terms of the $\ordfpp$ sequence are
\[(bp^{k}, bp^k+(k-1), bp^k+(k-2), \ldots, bp^k+1, bp^{k}).\] 
It is now clear that the $\ordfpp$ sequence and the $\incfp$ sequence of $x$ are periodic of period $k$.
\end{proof}

We will see later that the periodicity predicted by Corollary \ref{corollary:basecase} will eventually happen as part of the $\ordfpp$ sequence of $x$ for all nonzero $x \in K$ as long as $\ordfpp(\ordfpp(x)) \geq 0$ and the $\ordfpp$ sequence of $x$ is not eventually $+\infty$.

\begin{definition}
    For any integer $k \geq 1$, we call the following sequence
    \[\mathcal{S}_{k,p} := (k-1, \underbrace{-1, -1, \ldots, -1}_{(k-1 \bmod p) \text{ copies}})\]
    the \emph{$k$-segment} (with respect to $p$).
\end{definition}

We define a sequence of integers $\kappa_0, \kappa_1, \kappa_2, \ldots$ recursively from $\ordfpp(x)$ that will allow us to predict the period of the $\ordfpp$ sequence of $x$. Let $\kappa_0 := \ordfpp(x) \bmod p$ and $\kappa_1 := \ordp(\pfloor{\kappa_0})$. Here $\pfloor{x} := x - (x \bmod p)$. For $i \geq 2$, we define
\begin{equation} \label{eqn:steps}  
\kappa_{i} := \begin{cases}
\nu_{p}(\pfloor{\kappa_{i-1}-1}), & \text{if $\kappa_{i-1} < +\infty$};\\
+\infty, & \text{if $\kappa_{i-1} = +\infty$}.
\end{cases}
\end{equation}

It is clear that if $1 \leq \kappa_i \leq p$, then $\kappa_{i+1} = +\infty$; if $p+1 \leq \kappa_i < +\infty$, then $\kappa_{i+1} < \log_p(\kappa_i)$. If the $\ordfpp$ sequence of $x$ is not eventually $+\infty$, then there exists a unique positive integer $N = N(k)$ such that $1 \leq \kappa_N \leq p$, and $\kappa_i = +\infty$ for all $i > N$. 

\begin{theorem} \label{theorem:main1}
Let $x \in K$ be a nonzero element such that $\ordfpp(x) = bp^{k}$ with $\ordp(b) = 0$ and $k \geq 0$.    If the $\ordfpp$ sequence of $x$ is not eventually $+\infty$, then the $\incfp$ sequence of $x$ is of the form 
\[(\underbrace{-1, -1, \ldots, -1}_{\kappa_0 \text{ copies}}, \mathcal{S}_{\kappa_1, p}, \mathcal{S}_{\kappa_2,p}, \mathcal{S}_{\kappa_3,p}, \ldots, \mathcal{S}_{\kappa_N,p}, \mathcal{S}_{\kappa_N,p}, \mathcal{S}_{\kappa_N,p}, \ldots).\]
As a result, the $\ordfpp$ sequence and the $\incfp$ sequence of $x$ are eventually periodic of period $\kappa_{N}$.
\end{theorem}
\begin{proof}
For $0 \leq i \leq \kappa_0$, we have $\ordfpp(\dKp^i(x)) = b-i = \ordfpp(x) - i$. Hence the first $\kappa_0$ terms of the $\incfp$ sequence of $x$ are
\[(\underbrace{-1, -1, \ldots, -1}_{\kappa_0 \text{ copies}}).\]
We can write $\ordfpp(\dKp^{\kappa_0}(x)) = b_0p^{\kappa_1}$ with $\kappa_1 \geq 1$. By Proposition \ref{prop:segment}, we know that the next $\kappa_1' := (\kappa_1-1 \bmod p)+1$ term of the $\incfp$ sequence is the $\kappa_1$-segment
\[\mathcal{S}_{\kappa_1,p} = (\kappa_1 - 1, \underbrace{-1, -1, \ldots, -1}_{(\kappa_1-1) \bmod p \text{ copies}}).\]
Furthermore, we get $\ordfpp(\dKp^{\kappa_0+i}(x)) = bp^{\kappa_1} + (\kappa_1 - i)$ for $0 \leq i \leq \kappa_1'$. As $\kappa_1 - \kappa_1'  = \pfloor{\kappa_1 -1}$ and $\kappa_2 = \ordp(\pfloor{\kappa_1-1})$, we can write $\ordfpp(\dKp^{\kappa_0+\kappa_1'+1}(x)) = b_1p^{\kappa_2}$. If $\kappa_2 \geq 1$, by Proposition \ref{prop:segment} again, we know that the next $\kappa_2' := (\kappa_2-1 \bmod p)+1$ term of the $\incfp$ sequence is the $\kappa_2$-segment. Let $N = N(k)$ be the unique positive integer such that $1 \leq k_N \leq p$. By induction, we know that the initial terms of the $\incfp$ sequence of $x$ is of the form
\[(\underbrace{-1, -1, \ldots, -1}_{\kappa_0 \text{ copies}}, \mathcal{S}_{\kappa_1,p}, \mathcal{S}_{\kappa_2,p}, \mathcal{S}_{\kappa_3,p}, \ldots, \mathcal{S}_{\kappa_N,p}).\]
Corollary \ref{corollary:basecase} implies that if $b_{N-1}p^{\kappa_N}$ is a term in the $\ordfpp$ sequence of $x$, then $\mathcal{S}_{k_N,p}$ will appear repeatedly in the $\incfp$ sequence of $x$. 
\end{proof}

\subsection{Anti-partial derivatives}

We fix a finite extension $K/\mathbb{Q}_p$ in this subsection. Note that not all elements in $K$ have an anti-partial derivative. For example, suppose $x \in K$ is an anti-partial derivative of $p^{p-1} \in K$, then $\dKp^{p+1}(x) = 0$ and thus the $\ordfpp$ sequence of $x$ is eventually $+\infty$. By Lemma \ref{lemma:eventuallyinfinity}, $\ordfpp(x) \in \{0, 1, 2, \dots, p-1, +\infty\}$, but that is not possible as $\ordfpp(\dKp(x)) = p-1$. Therefore $p^{p-1}$ does not have anti-partial derivative in $K$. Given an element $y \in K$, if $y$ has an anti-partial derivative in $K$, we want to know how many there are. We start with $y=0$. Let $x \in K$ such that 
\[\dKp(x) = \frac{x\ordfpp(x)}{p} = 0.\]
Then $x\ordfpp(x)= 0 $ which implies that $x = 0$ or $\ordfpp(x) = 0$. Hence the anti-partial derivative of $0$ in $K$ is
\[\{x \in K \, : \, \ordfpp(x) = 0\} \, \cup \, \{0\}.\]

\begin{lemma} \label{lemma:antidev}
For every $0 \neq y \in K$, if there exists $x \in K$ such that $\dKp(x) = y$, then $x \in \mathbb{Q}_p(y)$.
\end{lemma}
\begin{proof}
Since $\dKp(x) = x\ordfpp(x)/p = y$ and $\ordfpp(x)/p \in \mathbb{Q}$, we know that $x \in \mathbb{Q}_p(y)$.
\end{proof}

Let $x_1, x_2 \in K$ with $\dKp(x_1) = \dKp(x_2)$. If $\ordfpp(x_1) = 0$, then $\dKp(x_1) = 0 = \dKp(x_2)$. Thus $\ordfpp(x_2) = 0$. Hence $\ordfpp(x_1) = 0$ if and only if $\ordfpp(x_2) = 0$.

Suppose $\ordfpp(x_1), \ordfpp(x_2) \neq 0$. Let $\ordfpp(x_1) = b_1p^{k_1}$ and $\ordfpp(x_2) = b_2p^{k_2}$ where $\ordp(b_1b_2) = 0$. We get 
\begin{equation} \label{eq:bk4}
    b_1p^{k_1} - b_2p^{k_2} = k_2-k_1.
\end{equation}

Suppose $k_1 = k_2$, then \eqref{eq:bk4} implies that $\ordfpp(x_1) = \ordfpp(x_2)$. Hence
\[x_1 = \dfrac{\dKp(x_1)p}{\ordfpp(x_1)} = \dfrac{\dKp(x_2)p}{\ordfpp(x_2)} = x_2.\]
This means that $x_1 = x_2$ if and only if $k_1 = k_2$.

If $k_1 \neq k_2$, without loss of generality, we assume $k_1 < k_2$. Suppose $k_1 < 0$, then \eqref{eq:bk4} implies that
\[b_1 - b_2p^{k_2 - k_1} = p^{-k_1}(k_2-k_1).\]
This is a contradiction because $\ordfpp(b_1 - b_2p^{k_2-k_1}) = 0$ and $\ordfpp(p^{-k_1}(k_2-k_1)) \geq -k_1 > 0$. Hence if $k_1 < 0$, then $\dKp(x_1)$ has exactly one anti-partial derivative.

Suppose $k_1 > 0$. There is an element $x_0 \in K$ in the set of all anti-partial derivatives of $\dKp(x_1)$ with the smallest possible $k_0$. We call $x_0$ the \emph{primitive} anti-partial derivative of $\dKp(x_1)$. Equation \eqref{eq:bk4} implies that

\begin{align}
    b_0p^{k_0} - bp^{k_1}  &= k_1 - k_0, \label{eq:bk6}
\end{align}
As $x_0$ is primitive, we have $k_0 \leq k_1$ and \eqref{eq:bk6} implies that $p^{k_0}(b_0-bp^{k_1-k_0}) = k_1-k_0$. Let $k_1-k_0=p^{k_0}c$ for some $c \in \mathbb{Z}_{\geq 0}$. Then $b_0-bp^{p^{k_0}c}= c$. So $b = \dfrac{b_0-c}{p^{p^{k_0}c}}$ and $\ordp(b_0-c) = p^{k_0}c$ since $\ordp(b)=0$. Let
\[C(x_0) := \Big\{c \in \mathbb{Z}_{\geq 0} \, :  \, \ordp(b_0-c) = p^{k_0}c \Big\}.\]
It is easy to see that $C(x_0)$ is finite because as $c \gg 0$, $\ordp(b_0-c) < p^{k_0}c$.

\begin{theorem} \label{theorem:primitive}
With the above notations, suppose $x_0$ is the primitive anti-partial derivative of $\dKp(x_0)$. Let $\ordfpp(x_0) = b_0p^{k_0}$ with $\ordp(b_0) = 0$ and $k_0 > 0$. There is a one-to-one correspondence between $C(x_0)$ and the set of all anti-partial derivatives of $\dKp(x_0)$. Furthermore, suppose we fix a uniformizer $\pi \in \mathfrak{p} \subset \mathcal{O}_K$ and let $e$ be the ramification index of $K/\mathbb{Q}_p$, we can write $x_0 = \alpha_0\pi^{eb_0p^{k_0}}$ and $p = \alpha_p \pi^e$ with $\alpha_0, \alpha_p \in \mathcal{O}_K^{\times}$. If $x = \alpha\pi^{ebp^{k}}$ is an anti-partial derivative of $\dKp(x_0)$ such that $\ordp(b) = 0$ and $\alpha \in \mathcal{O}_K^{\times}$, then there exists a unique $c \in C(x_0)$ such that
\[k = p^{k_0}c+k_0 \in \mathbb{Z}_{\geq 0}, \quad b = \frac{b_0-c}{p^{k-k_0}} = \frac{b_0-c}{p^{p^{k_0}c}}, \quad \alpha = \dfrac{\alpha_0b_0}{b}\alpha_p^{k_0-k} \in \mathcal{O}_K^{\times}.\] 
\end{theorem}
\begin{proof}
We show that every anti-partial derivative $x$ of $\dKp(x_0)$ is associated with a unique $c \in C(x_0)$. If $x = x_0$, then we associate $x$ with $c=0$. Suppose $x \neq x_0$. Let $\ordfpp(x) = bp^{k}$. Since $x_0$ is the primitive anti-partial derivative and $\ordfpp(x_0) \neq 0$, we know that $b \neq 0$ and $k > k_0$. Then $p^{k_0}(b_0 - bp^{k-k_0}) = k-k_0$ and thus $\ordp(k-k_0) = k_0$. Let $k-k_0 = p^{k_0}c$ where $c > 0$ and $\ordp(c) = 0$. By plugging $k-k_0=p^{k_0}c$ into $p^{k_0}(b_0 - bp^{k-k_0}) = k-k_0$, we get $b_0 - bp^{k-k_0} = c$. Since $\ordp(b) = 0$, we know that $\ordp(b_0-c) = p^{k_0}c$.

Then we show that for each $c \in C(x_0)$, we can define a unique $x=x(c)$ such that $\dKp(x) = \dKp(x_0)$. Since $\ordp(b_0-c) = p^{k_0}c$, there exists $b \in \mathbb{Q}$ with $\ordp(b) = 0$ such that $b_0-c=bp^{p^{k_0}c}$. Set $k := p^{k_0}c+k_0$. We can compute
\begin{gather*}
bp^{k}+k-1 = \frac{b_0-c}{p^{k-k_0}} p^{k} + k -1 = (b_0-c)p^{k_0} + p^{k_0}c+k_0-1 \\
= (b_0-c)p^{k_0} + p^{k_0}c+k_0-1 = b_0p^{k_0} + k_0 -1.
\end{gather*}
Set $x:= \alpha \pi^{ebp^k}$ where $\alpha = \alpha_0b_0\alpha_p^{k_0-k}/b$. We have
\begin{gather*}
\dKp(x) = \dfrac{x\ordfpp(x)}{p} = \dfrac{\alpha\pi^{ebp^{k}}ebp^{k}}{p} = \alpha be\pi^{ebp^{k}}p^{k-1} = \alpha be\alpha_p^{k-1}\pi^{e(bp^{k}+k-1)} \\
= \alpha_0b_0e\alpha_p^{k_0-1}\pi^{e(b_0p^{k_0}+k_0-1)} = \dfrac{\alpha_0b_0e}{p}\pi^{eb_0p^{k_0}}p^{k_0} = \dfrac{x_0\ordfpp(x_0)}{p} = \dKp(x_0). \qedhere
\end{gather*}
\end{proof}

\begin{corollary}
For any nonzero $y\in K$, the set $\{x \in K : \dKp(x) = y\}$ is finite (possibly empty).
\end{corollary}

For the rest of this subsection, we will prove Conjecture \ref{conjecture:ufnarovski} for partial derivatives over any finite extension $K/\mathbb{Q}_p$. We will show that for each positive integer $n$, there exists infinitely many $x \in \mathbb{Q}_p$ such that $D_{\mathbb{Q}_p,p}(x)$ has exactly $n$ anti-partial derivatives in $\mathbb{Q}_p$. By Lemma \ref{lemma:antidev}, we know that all anti-partial derivatives of $D_{\mathbb{Q}_p,p}(x)$ must be in $\mathbb{Q}_p$ and thus $D_{\mathbb{Q}_p,p}(x)$ has exactly $n$ anti-partial derivatives in any finite extension $K/\mathbb{Q}_p$. The first lemma gives us a way to construct $k_0 \in \mathbb{Z}_{> 0}$ such that if $\ordfpp(\ordfpp(x_0)) = k_0$, then $x_0$ is the primitive anti-partial derivative of $\dKp(x_0)$.

\begin{lemma} \label{lemma:define_k}
For every integer $m \geq 2$, let $k_0 = k_0(m) := p + p^2  + \cdots + p^m$. For every $x_0 \in \mathbb{Q}_p$, if $\ordfpp(\ordfpp(x_0)) = k_0$, then $x_0$ is the primitive anti-partial derivative of $D_{\mathbb{Q}_p,p}(x_0)$.
\end{lemma}
\begin{proof}
Suppose $x_0$ is not the primitive anti-partial derivative of $D_{\mathbb{Q}_p,p}(x_0)$. Let $x \neq x_0$ be another anti-partial derivative of $D_{\mathbb{Q}_p,p}(x_0)$ with $\ordfpp(x) = bp^{k}$ such that $k < k_0$. If $k < 0$, we know that $D_{\mathbb{Q}_p,p}(x_0)$ has exactly one anti-partial derivative. Hence $k \geq 0$. Since $D_{\mathbb{Q}_p,p}(x) = D_{\mathbb{Q}_p,p}(x_0)$, we get $bp^k-b_0p^{k_0}=k_0-k$. This means that $\ordp(k_0-k) = k$. It suffices to show that no $0 \leq k < k_0$ satisfies this relation. It is clear that $k \neq 0$ because $\ordp(k_0) = 1$, and $k \neq 1$ because $\ordp(k_0-1) = 0$. Suppose $k > 1$. If $\nu_p(k_0 - k') = k$ for some $k' > 0$, then $k' \geq p + \cdots + p^{k-1} > k$. Therefore there does not exist an anti-partial derivative $x$ with $k<k_0$. This means that $x_0$ is the primitive anti-partial derivative of $D_{\mathbb{Q}_p,p}(x_0)$.
\end{proof}

The next lemma allows us to construct $b_0$ for every $k_0 > 0$ such that there are exactly $n-1$ different possible values of $c \in \mathbb{Z}_{>0}$ such that $\nu_p(b_0-c) = p^{k_0}c$. This means that the set $C(x_0)$  has exactly $n$ elements (with $0$ included).

\begin{lemma} \label{lemma:define_b}
Fix a positive integer $k$. Let $c_1 = 0$, and for $i \geq 2$, let $c_i := p^{p^{k}c_{i-1}} + c_{i-1}$. Suppose 
\[ C_n := \{c \in \mathbb{Z}_{>0} \; : \; \ordp(c_{n+1}-c) = p^{k}c\}. \]
Then $C_n = \{c_2, \ldots, c_n\}. $
\end{lemma}

\begin{proof}
We first note that for any $1 \leq i < j$, 
\[ c_j - c_i = \sum_{m = i}^{j-1} \left( c_{m+1} - c_m \right) = \sum_{m = i}^{j-1} p^{p^{k}c_m} \]
and so $\ordp(c_j - c_i) = p^kc_i$.  This shows that $c_m \in C_n$ if and only if $m \in \{2, 3, \ldots, n\}$.  

Next, we show that no other integers are in $C_n$. 
If  $c \in C_n$ where $c > c_{n+1}$, then $c - c_{n+1} = \alpha p^{p^{k}c}$, where $\alpha > 0$.
By definition of $c_{n+1}$, $c  - c_{n+1} = c - (c_n + p^{p^{k}c_n})$. Thus
\[ c - c_n  = \alpha\,p^{p^{k}c} + p^{p^{k}c_n} = p^{p^{k}c_n} \left( \alpha \, p^{p^{k}(c - c_n)} + 1 \right). \]
This is a contradiction, since the expression on the right hand side is clearly larger than $c - c_n$. This shows that if $c \in C_n$, then $c \leq c_{n+1}$.

Suppose $c \in C_n$ where $c_m < c < c_{m+1}$ for some $2 \leq m \leq n$. We have $\ordp(c_{n+1} - c_{m+1}) = p^{k}c_{m+1}$ when $m < n$. Since $\ordp(c_{n+1} - c) = p^{k}c$, we have 
\[ \ordp(c_{m+1} - c) = \ordp\Big( (c_{n+1} - c) - (c_{n+1} - c_{m+1}) \Big) = p^{k}c.  \]
Therefore $c_{m+1}  - c = \gamma p^{p^{k}c}$ for some $\gamma > 0$. By definition, $c_{m+1} = p^{p^{k}c_m} + c_m$, and so we would have 
\[ p^{p^{k}c_m} + c_m = \gamma p^{p^{k}c}  + c, \]
which is a contradiction, since the left side is clearly less than the right. This shows that if $c \in C_n$ and $c \leq c_{n+1}$, then $c = c_m$ for some $2 \leq m \leq n$. This concludes the proof of the lemma.
\end{proof}

\begin{theorem} \label{theorem:infinite_derivative}
For each positive integer $n$, there are infinitely many $x_0 \in K$ such that $\dKp(x_0)$ has exactly $n$ anti-partial derivatives in $K$.
\end{theorem}
\begin{proof}
By Lemma \ref{lemma:antidev}, it suffices to assume that $K = \mathbb{Q}_p$ and $\mathfrak{p} = (p)$. For every integer $m \geq 2$, let $k_0 = k_0(m)$ be defined as in Lemma \ref{lemma:define_k}, and let $b_0 = c_{n+1}$ be defined as in Lemma \ref{lemma:define_b} for $k = k_0$. Set $x_0 := p^{b_0p^{k_0}}$. Lemma \ref{lemma:define_k} implies that $x_0$ is the primitive anti-partial derivative of $D_{\mathbb{Q}_p,p}(x_0)$. Lemma \ref{lemma:define_b} implies that $D_{\mathbb{Q}_p,p}(x_0)$ has exactly $n$ anti-partial derivatives. Therefore, for each positive integer $n$, there exists infinitely many $x_0 \in \mathbb{Q}_p$ such that $D_{\mathbb{Q}_p,p}(x_0)$ has exactly $n$ anti-partial derivatives with $x_0$ being its primitive anti-partial derivative.
\end{proof}

\section{Number Fields} \label{section:numberfields}

In this section, we will generalize arithmetic derivative and arithmetic partial derivative to number fields. Recall that the explicit formula of the arithmetic derivative defined on $\mathbb{Q}$:
\[D_{\mathbb{Q}}(x) = x \sum_{p \mid x} \dfrac{\ordp(x)}{p}.\]
Let $K/\mathbb{Q}$ be a number field of finite degree. One could mimic the above formula and define the arithmetic derivative on $K$ by the formula:
\[D_{K}(x) = x \sum_{\mathfrak{p} \mid x} \dfrac{\ordfpp(x)}{p},\]
where $\mathfrak{p}$ are prime ideals in $\mathcal{O}_K$. The sum is finite as there are finitely many $\mathfrak{p}$ such that $\ordfpp(x) \neq 0$. This formula presents a challenge. Let $p \in \mathbb{Q}$ be a rational prime. Then 
\[D_K(p) = p \sum_{\mathfrak{p} \mid p} \dfrac{\nu_{\mathfrak{p}}(p)}{p} = \sum_{\mathfrak{p} \mid p} \nu_{\mathfrak{p}}(p) = g(p,K) \cdot 1 = g(p,K),\]
where $g(p,K)$ is the number of prime ideals in $\mathcal{O}_K$ that divide $p$. When $g(p,K) \neq 1$, $D_K(p) \neq D_{\mathbb{Q}}(p)$ so the above formula of $D_K$ does not give a true extension of $D_{\mathbb{Q}}$. In order for $D_K(x) = D_{\mathbb{Q}}(x)$ for all $x \in \mathbb{Q}$, we will need to divide $g(p,K)$. Furthermore, let $\mathfrak{p}_1$ and $\mathfrak{p}_2$ be two prime ideals in $\mathcal{O}_K$ that divide $p$ and let $L/K$ be a finite extension. We know that in general $g(\mathfrak{p}_1, L) \neq g(\mathfrak{p}_2, L)$ unless $L/K$ is finite Galois. So we will start the generalization of $D_{\mathbb{Q}}$ to finite Galois extensions. Then we can further generalize $D_{\mathbb{Q}}$ to number fields by taking restriction.

\subsection{Finite Galois extensions}
Let $K$ be a finite Galois extension of $\mathbb{Q}$ of degree $n$. Let $\mathcal{O}_K$ be the ring of integers and $\mathfrak{p}$ a nonzero prime ideal of $\mathcal{O}_K$ such that $p \in \mathfrak{p}$. There is a discrete valuation $\ordfpp$ on $K$ that extends the $p$-adic valuation $\nu_p$ on $\mathbb{Q}$. This induces a norm $|\cdot|_{\ordfpp} = [\mathcal{O}_K:\mathfrak{p}]^{-\ordfpp(\cdot)}$ on $K$. Let $K_{\ordfpp}$ be the completion of $K$ with respect to the $\mathfrak{p}$-adic topology and thus $K_{\ordfpp}$ is a finite extension of $\mathbb{Q}_p$ such that $\mathfrak{p} \cap \mathbb{Q} = (p)$ (denoted by $\mathfrak{p} \mid p$). Let $e(K_{\ordfpp}|\mathbb{Q}_p)$ be the ramification index and $f(K_{\ordfpp}|\mathbb{Q}_p) := [\mathcal{O}_K/\mathfrak{p} : \mathbb{F}_p]$ be the inertia degree of the extension $K_{\ordfpp}/\mathbb{Q}_p$. One has the following decomposition:
\[p\mathcal{O}_K = \prod_{\mathfrak{p} \mid p} \mathfrak{p}^{e(K_{\ordfpp}|\mathbb{Q}_p)}.\]
It is well known that for every fixed prime number $p$, we have the formula
\begin{equation} \label{eq:sum1}
n = \sum_{\mathfrak{p} \mid p} e(K_{\ordfpp}|\mathbb{Q}_p) f(K_{\ordfpp}|\mathbb{Q}_p).
\end{equation}
The Galois group $G(K/\mathbb{Q})$ acts transitively on the set of prime ideals $\{\mathfrak{p} \subset \mathcal{O}_K : \mathfrak{p} \mid p\}$ for every fixed prime $p \in \mathbb{Q}$ \cite[Chapter 1, Section 7, Proposition 19]{localfields}. This implies that for every nonzero prime ideal $\mathfrak{p} \mid p$, the ramification index $e(K_{\ordfpp}|\mathbb{Q}_p)$ and the inertia degree $f(K_{\ordfpp}|\mathbb{Q}_p)$ depend only on $p$. If we denote them by $e(p, K)$ and $f(p, K)$ respectively, then formula \eqref{eq:sum1} becomes
\begin{equation} \label{eq:sum2}
n = e(p, K) f(p, K) g(p,K),
\end{equation}
where $g(p,K)$ (again only depends on $p$) is the number of distinct prime ideals $\mathfrak{p}$ such that $\mathfrak{p} \mid p$. Now we can extend the arithmetic derivative $D_{\mathbb{Q}}$ to $K$. For every nonzero $x \in K$, we define
\[D_K(x) := x \sum_{\mathfrak{p} \mid x} \dfrac{\ordfpp(x)}{pg(p,K)}.\]
One can check that $D_K$ satisfies the Leibniz rule:
\begin{align*}
D_K(xy) &= xy \sum_{\mathfrak{p} \mid xy} \dfrac{\ordfpp(xy)}{pg(p,K)} \\
&= xy \sum_{\mathfrak{p} \mid xy} \dfrac{\ordfpp(x)+\ordfpp(y)}{pg(p,K)} \\
&= \Big(\sum_{\mathfrak{p} \mid xy} \dfrac{x\ordfpp(x)}{pg(p,K)}\Big)y + x\Big(\sum_{\mathfrak{p} \mid xy} \dfrac{y\ordfpp(y)}{pg(p,K)}\Big)\\
&= \Big(\sum_{\mathfrak{p} \mid x} \dfrac{x\ordfpp(x)}{pg(p,K)}\Big)y + x\Big(\sum_{\mathfrak{p} \mid y} \dfrac{y\ordfpp(y)}{pg(p,K)}\Big) \\
&= D_K(x)y+xD_K(y).
\end{align*}
It is easy to check that $D_K(0) = 0$. To check that $D_K: K \to K$ extends $D_{\mathbb{Q}}: \mathbb{Q} \to \mathbb{Q}$, recall that for every prime $p$, we have $\ordfpp(x) = \ordp(x)$ for every $x \in \mathbb{Q}$. And so for every nonzero $x \in \mathbb{Q}$, we get
\[D_K(x) = x \sum_{\mathfrak{p} \mid x} \dfrac{\ordfpp(x)}{pg(p,K)}  = x \sum_{p \mid x} \Big(\dfrac{g(p,K) \cdot \ordp(x)}{pg(p,K)} \Big)  = x \sum_{p \mid x} \dfrac{\ordp(x)}{p} = D_{\mathbb{Q}}(x).\]

\subsection{Number fields} 
Let $K/\mathbb{Q}$ be a number field and let $L/K$ be an extension such that $L/\mathbb{Q}$ is Galois (e.g., one can take $L$ to be a Galois closure of $K/\mathbb{Q}$). For every $x \in K$, one can define $D_K(x) = D_L(x)$. But we want to make sure that $D_L(x) = D_K(x)$ for all $x \in K$ so the definition of $D_K$ does not depend on the choice of Galois extensions.

\begin{lemma}
Suppose $K/\mathbb{Q}$ and $L/\mathbb{Q}$ are finite Galois extensions. We have $D_K(x) = D_L(x) $ for every $x \in K \cap L$.
\end{lemma}
\begin{proof}
We first assume that $K \subset L$. Since $L/\mathbb{Q}$ is Galois, we know that $L/K$ is also Galois. For every rational prime $p$ and nonzero prime ideals $\mathfrak{p}_1$ and $\mathfrak{p}_2$ of $\mathcal{O}_K$ with $\mathfrak{p}_1 \mid p$ and $\mathfrak{p}_2 \mid p$, we get $g(\mathfrak{p}_1, L) = g(\mathfrak{p}_2,L)$. Let $\mathfrak{p}$ and $\mathfrak{P}$ be two prime ideals in $\mathcal{O}_K$ and $O_L$ respectively such that $\mathfrak{P} \mid \mathfrak{p} \mid p$.
For every nonzero $x \in K$, we have
\begin{align*}
D_L(x) &= x \sum_{\mathfrak{P} \mid x} \dfrac{\ordfPP(x)}{pg(p,L)} = x \sum_{\mathfrak{p} \mid x} \sum_{\mathfrak{P} \mid \mathfrak{p}} \dfrac{\ordfpp(x)}{pg(p,L)} \\
&= x \sum_{\mathfrak{p} \mid x} \dfrac{g(\mathfrak{p},L)\ordfpp(x)}{pg(\mathfrak{p}, L)g(p,K)} = x \sum_{\mathfrak{p} \mid x} \dfrac{\ordfpp(x)}{pg(p,K)} = D_K(x).
\end{align*}
This shows that $D_K(x) = D_L(x)$ for all $x \in K$ if $K \subset L$.

Now suppose $K/\mathbb{Q}$ and $L/\mathbb{Q}$ are two arbitrary finite Galois extensions. Since $K \cap L$ is also a finite Galois extension of $\mathbb{Q}$, for every $x \in K \cap L$, we have $D_K(x) = D_{K \cap L}(x)$ by the previous paragraph. Using the same argument, we get $D_L(x) = D_{K \cap L}(x)$ for every $x \in K \cap L$, and therefore $D_K(x) = D_L(x)$ for every $x \in K \cap L$.
\end{proof}

Suppose $K/\mathbb{Q}$ is a number field (not necessarily Galois). For every $x \in K$, we can define $D_K(x) := D_{K^{\textrm{Gal}}}(x)$ where $K^{\textrm{Gal}}$ is a Galois closure of $K/\mathbb{Q}$. When $x \neq 0$, it is clear that $D_K(x)/x \in \mathbb{Q}$ and thus $D_K(x) \in K$. We have a well-defined arithmetic derivative $D_K: K \to K$ when $K$ is a number field.

\subsection{Arithmetic subderivative}
Let $S$ be a (finite or infinite) subset of the prime numbers $\mathbb{P}$. One can define the so-called arithmetic subderivative $D_{\mathbb{Q}, S}: \mathbb{Q} \to \mathbb{Q}$ by 
\[D_{\mathbb{Q}, S}(x) = \sum_{p \in S} x\ordp(x)/p.\] It is easy to see that $D_{\mathbb{Q}, S} = \sum_{p \in S} D_p$ and $D_{\mathbb{Q}} = \sum_{p \in \mathbb{P}} D_p$. One can extend $D_{\mathbb{Q}, S}$ to all finite Galois extensions $K/\mathbb{Q}$. Let $T$ be a set of prime ideals of $\mathcal{O}_K$. For every nonzero $x \in K$, we define
\[D_{K, T}(x): = x \sum_{\mathfrak{p} \in T, \mathfrak{p} \mid p} \dfrac{\ordfpp(x)}{pg(p,K)}.\]
If $T = \{\mathfrak{p}\}$ contains only one prime ideal, then we call $D_{K,T} = \dKp$ the arithmetic partial derivative with respect to $\mathfrak{p}$. By taking $K = \mathbb{Q}$ and $\mathfrak{p} = \{p\}$, we can see $\dKp$ is the generalization of arithmetic partial derivative with respect to $p$. Suppose $L/K$ is a finite Galois extension. Let
\[T_{L/K} = \{\mathfrak{P} : \mathfrak{P} \textrm{ prime ideal of } \mathcal{O}_L, \exists \; \mathfrak{p} \in T \textrm{ such that } \mathfrak{P} \mid \mathfrak{p}\}.\]
For every nonzero $x \in K$, we have
\begin{align*}
D_{L,T_{L/K}}(x) &= \sum_{\mathfrak{P} \in T_{L/K}, \mathfrak{P} \mid p} \dfrac{x\nu_{\mathfrak{P}}(x)}{pg(p,L)} = \sum_{\mathfrak{p} \in T} \sum_{\mathfrak{P} \in T_{L/K}, \mathfrak{P} \mid \mathfrak{p}} \dfrac{x\nu_{\mathfrak{P}}(x)}{pg(p,L)} \\
&= \sum_{\mathfrak{p} \in T} g(\mathfrak{p}, L) \dfrac{x\ordfpp(x)}{pg(p,K)g(\mathfrak{p}, L)} = \sum_{\mathfrak{p} \in T} \dfrac{x\ordfpp(x)}{pg(p,K)} = D_{K,T}(x).
\end{align*}
In this case, $D_{L,T_{L/K}}$ extends $D_{K,T}$.

If $K/\mathbb{Q}$ is a number field (not necessarily Galois), we can define $D_{K,T}$ via a larger Galois extension. Let $L/K$ be a finite extension such that $L/\mathbb{Q}$ is Galois. Let $T_{L/K}$ be defined as above. We can define $D_{K,T}(x) := D_{L, T_{L/K}}(x)$ for all $x \in K$. Again this definition does not depend on the choice of Galois extensions. Let $L_1/K$ and $L_2/K$ be finite extensions such that $L_1/\mathbb{Q}$ and $L_2/\mathbb{Q}$ are Galois. Let $L_3 := L_1 \cap L_2$ and $T' := T_{L_3 / K}$. We note that $T_{L_1/K} = T'_{L_1/L_3}$ and $T_{L_2/K} = T'_{L_2/L_3}$. Therefore for every $x \in K \subset L_3$, we have
\[D_{L_1, T_{L_1/K}}(x) = D_{L_1, T'_{L_1/L_3}}(x) = D_{L_3, T'}(x) = D_{L_2, T'_{L_2/L_3}}(x) = D_{L_2, T_{L_2/K}}(x).\]

\begin{remark}
Let $K/\mathbb{Q}$ be a finite Galois extension. Just like in the local case, one can ask whether Theorems \ref{theorem:main0} and \ref{theorem:main2} are true for $\dKp$. Note that in the global case $\dKp(x) = \frac{x\ordfpp(x)}{pg(p,K)}$, whereas in the local case $g(p,K) = 1$. If $\ordp(g(p,K)) = 0$, then $\ordfpp(\dKp(x)) = \ordfpp(x) + \ordfpp(\ordfpp(x)) - 1$, which is the same as Equation \eqref{equation:ordp_inc}. In this case, Theorems \ref{theorem:main0} and \ref{theorem:main2} are still true and can be proved in a similar fashion. If  $\ordp(g(p,K)) = a > 0$, then $\ordfpp(\dKp(x)) = \ordfpp(x) + \ordfpp(\ordfpp(x)) - 1 -a$. In this case, the behavior of the $\ordfpp$ sequence of $x$ warrants further study.
\end{remark}

\section{Arithmetic Logarithmic Derivative }

\subsection{Local Case}
The logarithmic partial derivative (with respect to $p$) $\ldQ : \mathbb{Q}^{\times} \to \mathbb{Q}$ is a homomorphism defined by the formula 
\[\ldQ(x)= D_{\mathbb{Q},p}(x)/x\]
because 
\[\ldQ(xy) = \dfrac{D_{\mathbb{Q},p}(xy)}{xy} = \dfrac{D_{\mathbb{Q},p}(x)y+xD_{\mathbb{Q},p}(y)}{xy} = \ldQ(x) + \ldQ(y).\]
The image of $\ldQ$ is 
\[\ldQ(\mathbb{Q}^{\times}) = \{m/p : m \in \mathbb{Z}\} = \langle 1/p \rangle \cong \mathbb{Z}\]
and thus $\ldQ$ is not onto. Suppose $\ldQ(x) = 0$, then $D_{\mathbb{Q},p}(x) = 0$ and thus $\ordp(x) = 0$. Therefore
\[\mathrm{Ker}(\ldQ) = \{x \in \mathbb{Q}^{\times} : \ordp(x) = 0\}.\]
One can extend $\ldQ$ to $\mathbb{Q}_p^{\times}$ by the formula $\ldQp(x) := D_{\mathbb{Q}_p,p}(x)/x \in \mathbb{Q}$. Using the same argument, we get 
\[\ldQp(\mathbb{Q}_p^{\times}) = \{m/p : m \in \mathbb{Z}\}, \qquad \mathrm{Ker}(\ldQp) = \{x \in \mathbb{Q}_p^{\times} : \ordp(x) = 0\}.\]
Let $K/\mathbb{Q}_p$ be a finite extension. We can define $\ldKp : K^{\times} \to \mathbb{Q}$ as 
\[\ldKp(x) := \frac{\dKp(x)}{x} = \dfrac{\ordfpp(x)}{p}.\]
It is easy to see the kernel of $\ldKp$ is
\[\mathrm{Ker}(\ldKp) = \{x \in K^{\times} : \ordfpp(x) = 0\}.\]
The description of the image of $\ldKp$ depends on whether $p$ divides the ramification index $e$. Let $e = p_1^{r_1}p_2^{r_2}\cdots p_j^{r_j}$ be the unique factorization of the ramification index into prime powers. If $p \notin \{p_1, p_2, \dots, p_j\}$, then
\[\ldKp(K^{\times}) = \{m/pe : m \in \mathbb{Z}\} = \langle 1/p, 1/p_1^{r_1}, \dots, 1/p_j^{r_j} \rangle \cong \mathbb{Z}.\]
If $p \in \{p_1, p_2, \dots, p_j\}$ and assume $p = p_1$, then
\[\ldKp(K^{\times}) = \{m/pe : m \in \mathbb{Z}\} = \langle 1/p_1^{r_1+1}, 1/p_2^{r_2}, \dots, 1/p_j^{r_j} \rangle \cong \mathbb{Z}.\]

\subsection{Global Case}
If $K/\mathbb{Q}$ is a finite Galois extension, one can define the arithmetic logarithmic derivative $\ld_K: K^{\times} \to \mathbb{Q}$ as 
\[\ld_K(x) = \dfrac{D_K(x)}{x} = \sum_{\mathfrak{p} \mid x} \dfrac{\ordfpp(x)}{pg(p, K)} \in \mathbb{Q}.\]
It is easy to show that $\ld_K$ is a group homomorphism. When $K = \mathbb{Q}$, we get that $\ld_{\mathbb{Q}}(x) = \sum_{p \mid x} \frac{\ordp(x)}{p}$. Hence $\ld_{\mathbb{Q}}(\mathbb{Q}^{\times}) = \langle \frac{1}{p} : p \in \mathbb{P} \rangle$. For every finite Galois extension $K/\mathbb{Q}$, one can show that $\ld_K(K^{\times})$ are isomorphic as subgroups of $\mathbb{Q}$. Before we prove this result, we need to recall a concept called $p$-height in the classification of subgroups of $\mathbb{Q}$. Let $G$ be an (additive) subgroup of $\mathbb{Q}$ and $g \in G$. The $p$-height of $g$ in $G$ is $k$ if $p^kx=g$ is solvable in $G$ and $p^{k+1}x = g$ is not. If $p^kx=a$ has a solution for every $k$, then we say that the $p$-height of $a$ in $G$ is infinite. Let $H_{p_i,G}(g)$ be the $p_i$-height of $g$ in $G$. Set $H_G(g) := (H_{2,G}(g), H_{3,G}(g), H_{5,G}(g), \ldots)$. It turned out that $H_G(1)$ is an invariant of the subgroup $G$ in the following sense.

\begin{theorem}{\cite[Theorem 4]{ratiso}} \label{theorem:ratiso}
Let $G_1$ and $G_2$ be two subgroups of $\mathbb{Q}$. Then $G_1 \cong G_2$ if and only if $H_{G_1}(1)$ and $H_{G_2}(1)$ only differ in finitely many indices, and in the case $H_{p_i,G_1}(1) \neq H_{p_i,G_2}(1)$, both of them are finite.
\end{theorem}

\begin{theorem} \label{theorem:galiso}
Let $K/\mathbb{Q}$ be a finite Galois extension. Then $\ld_{K}(K^{\times}) \cong \langle \frac{1}{p} : p \in \mathbb{P} \rangle < \mathbb{Q}$.
\end{theorem}
\begin{proof}
Let $G := \langle \frac{1}{p} : p \in \mathbb{P} \rangle < \mathbb{Q}$. It is easy to see that
\[H_G = (1, 1, 1, \ldots).\]

Let $[K:\mathbb{Q}] = n$ and $\ordfp(x) := \ordfpp(x) e(p,K)$ be the normalized discrete valuation. For every $x \in K^{\times}$, we have
\[\ld_K(x) =  \sum_{\mathfrak{p} \mid x} \dfrac{\ordfpp(x)}{pg(p, K)} = \sum_{\mathfrak{p} \mid x} \dfrac{\ordfp(x)}{pg(p, K)e(p, K)} = \dfrac{1}{n} \sum_{\mathfrak{p} \mid x} \dfrac{\ordfp(x)f(p, K)}{p}.\]
Therefore
\begin{align*}
\ld_K(K^{\times}) &= \Big\{ \dfrac{1}{n} \sum_{\mathfrak{p} \mid x} \dfrac{\ordfp(x)f(p, K)}{p} \mid x \in K^{\times} \Big\} \\
&= \Big\langle \dfrac{f(p,K)}{np} \mid p \in \mathbb{P} \Big\rangle \\
&= \Big\langle \dfrac{1}{p^{1+\ordp(n)-\ordp(f(p,K))}} \mid p \in \mathbb{P} \Big\rangle.
\end{align*}
For every $p \in \mathbb{P}$, we denote $m(p) := 1+\ordp(n)-\ordp(f(p,K))$. It is easy to see that
\[H_{\ld_K(K^{\times})} = (m(2), m(3), m(5), \ldots).\]
As $f(p,K) \mid n$, we know that $1 \leq m(p) < +\infty$. When $p > n$, we have $\ordp(n) = \ordp(f(p,K)) = 0$. This implies that $m(p) = 1$ for all except for finitely many primes. Hence $H_G$ and $H_{\ld_K(K^{\times})}$ only differ in finitely many indices, and in the case $H_{p_i,G}(1) \neq H_{p_i,\ld_K(K^{\times})}$, both of them are finite. Hence $\ld_K(K^{\times}) \cong G$ by Theorem \ref{theorem:ratiso}.
\end{proof}

To determine the exact image of $\ld_K$ in general is not easy. We give an example.

\begin{example}
Let $K = \mathbb{Q}(\sqrt{D})$ be a quadratic extension, where $D$ is a square free integer. We rewrite the formula of $\ld_K$ using the normalized discrete valuation $\ordfp = \ordfpp \cdot e(p, K)$
\[\ld_K(x) = \sum_{\mathfrak{p} \mid x} \dfrac{\ordfpp(x)}{pg(p, K)} = \sum_{\mathfrak{p} \mid x} \dfrac{\ordfp(x)}{pg(p, K)e(p, K)}  = \dfrac{1}{2} \sum_{\mathfrak{p} \mid x} \dfrac{\ordfp(x)f(p, K)}{p}.\]
It remains to determine when $2$ is inert in $K$, that is, $f(2, K) = 2$. Let $\Delta_K$ be the discriminant of $K$, that is, $\Delta_K = D$ if $D \equiv 1 \pmod 4$ and $\Delta_K =4D$ if $D \equiv 2, 3 \pmod 4$. Hence $\Delta_K \equiv 0, 1, 4, 5 \pmod 8$. We know that $\mathcal{O}_K = \mathbb{Z}[\frac{\Delta_K+\sqrt{\Delta_K}}{2}]$. The minimal polynomial of $\frac{\Delta_K+\sqrt{\Delta_K}}{2}$ is
\[(X- \frac{\Delta_K+\sqrt{\Delta_K}}{2})(X - \frac{\Delta_K-\sqrt{\Delta_K}}{2}) = X^2 - \Delta_K X + \dfrac{\Delta_K^2 - \Delta_K}{4}.\]
We discuss the cases based on the value of $\Delta_K \bmod 8$.
\begin{enumerate}
    \item If $\Delta_K \equiv 0 \pmod 8$, then $\Delta_K^2-\Delta_K \equiv 8 \pmod 8$. Hence $\frac{\Delta_K^2-\Delta_K}{4} \equiv 0 \pmod 2$. Therefore
    \[X^2 - \Delta_K X + \dfrac{\Delta_K^2 - \Delta_K}{4} \equiv X^2 \pmod 2,\]
    and $(2) = (2, \frac{\Delta_K + \sqrt{\Delta_K}}{2})^2$ is ramified in this case, that is, $e(2, K) = 2$.
    \item If $\Delta_K \equiv 1 \pmod 8$, then $\Delta_K^2-\Delta_K \equiv 1 - 1 \equiv 0 \pmod 8$. Hence $\frac{\Delta_K^2-\Delta_K}{4} \equiv 0 \pmod 2$. Therefore
    \[X^2 - \Delta_K X + \dfrac{\Delta_K^2 - \Delta_K}{4} \equiv X^2 + X \equiv X(X+1)\pmod 2,\]
    and $(2) = (2, \frac{\Delta_K + \sqrt{\Delta_K}}{2})(2, \frac{\Delta_K + \sqrt{\Delta_K}}{2}+1)$ is totally split in this case, that is, $g(2, K) = 2$.
    \item If $\Delta_K \equiv 4 \pmod 8$, then $\Delta_K^2-\Delta_K \equiv 0 - 4 \equiv 4 \pmod 8$. Hence $\frac{\Delta_K^2-\Delta_K}{4} \equiv 1 \pmod 2$. Therefore
    \[X^2 - \Delta_K X + \dfrac{\Delta_K^2 - \Delta_K}{4} \equiv X^2 + 1 \equiv (X+1)^2 \pmod 2,\]
    and $(2) = (2, \frac{\Delta_K + \sqrt{\Delta_K}}{2}+1)^2$ is ramified in this case, that is, $e(2, K) = 2$.
    \item If $\Delta_K \equiv 5 \pmod 8$, then $\Delta_K^2-\Delta_K \equiv 1 - 5 \equiv 4 \pmod 8$. Hence $\frac{\Delta_K^2-\Delta_K}{4} \equiv 1 \pmod 2$. Therefore
    \[X^2 - \Delta_K X + \dfrac{\Delta_K^2 - \Delta_K}{4} \equiv X^2 + X + 1 \pmod 2,\]
    which is irreducible. In this case, $2$ is inert, that is, $f(2, K) = 2$.
\end{enumerate}
If $\Delta_K \equiv 5 \pmod 8$, then $\Delta_K \equiv 1 \pmod 4$. In this case, $\Delta_K = D$ and thus $D \equiv 5 \pmod 8$. Therefore
 \begin{displaymath}
    \ld_K(K^{\times}) = \begin{cases}
        \langle 1/2, 1/3, 1/5,  \ldots, \rangle, & \text{if $D \equiv 5 \pmod 8$};\\
        \langle 1/4, 1/3, 1/5, \ldots, \rangle, & \text{otherwise}.
        \end{cases}
 \end{displaymath}
\end{example}

\section{$p$-adic Continuity and Discontinuity}

In this section, we study when arithmetic partial derivatives and arithmetic subderivatives are $p$-adically continuous and discontinuous. When they are continuous, we will also study if they are strictly differentiable. We first recall some definitions.

Let $K$ be a field and $\nu: K \to \mathbb{R} \cup \{+\infty\}$ be a discrete valuation. For all $x, y \in K$, we have $\nu(x+y) \geq \mathrm{min}\{\nu(x), \nu(y)\}$. An important property of $\nu$ that we will use repeatedly in this subsection is that if $\nu(x) \neq \nu(y)$, then $\nu(x+y) = \mathrm{min}\{\nu(x), \nu(y)\}$. If $c$ is a real number number between $0$ and $1$, then the discrete valuation $\nu$ induces an absolute value on $K$ as follows:
\begin{displaymath}
    |x|_{\nu} := \begin{cases}
        c^{\,\nu(x)}, & \text{if $x \neq 0$};\\
        0, & \text{if $x=0$}.
        \end{cases}
\end{displaymath}
We then have the formula $|x+y|_{\nu} \leq \textrm{max}\{|x|_{\nu},|y|_{\nu}\}$ and thus $|\cdot|$ is an ultrametic absolute value. The subset $\mathcal{O}_K = \{x \in K : \nu(x) \geq 0\}$ is a ring with the unique maximal ideal $\mathfrak{p} = \{x \in K : \nu(x) > 0\}$. Let $f: K \to K$ be a function. We say that $f$ is $\mathfrak{p}$-adically continuous at a point $x \in K$ if for every $\epsilon > 0$, there exists $\delta > 0$ such that for every $|y-x|_{\nu} < \delta$, we have $|f(y)-f(x)|_{\nu} < \epsilon$. Equivalently, to show that $f$ is $\mathfrak{p}$-adically continuous at $x$, it is enough to show that for every sequence $x_i$, 
\[\lim_{i \to +\infty} \nu(x - x_i) = +\infty \quad \mathrm{implies} \quad \lim_{i \to +\infty} \nu(f(x) - f(x_i)) = +\infty.\] 
On the contrary, to show that $f$ is $\mathfrak{p}$-adically discontinuous at $x$, it is enough to find one sequence $x_i$ such that 
\[\lim_{i \to +\infty} \nu(x - x_i) = +\infty \quad \mathrm{and} \quad \lim_{i \to +\infty} \nu(f(x) - f(x_i)) \neq +\infty.\]

Recall that $f$ is differentiable at a point $x$ if the difference quotients $(f(y)-f(x))/(y-x)$ have a limit as $y \to x$ ($y \neq x$) in the domain of $f$. When the absolute value of the domain is ultrametric, we study the so-called strict differentiability. For more details on $p$-adic analysis, we refer the reader to \cite{p_adic_analysis}.

\begin{definition} \label{definition:diff}
Let $K$ be a field equipped with an ultrametric absolute value $|\cdot|_{\nu}$. We say that $f : K \to K $ is \emph{strictly differentiable} at a point $x \in K$ (with respect to $|\cdot|_{\nu}$) if the difference quotients \[\Phi f(u,v) = \dfrac{f(u)-f(v)}{u-v}\] have a limit as $(u,v) \to (x,x)$ while $u$ and $v$ remaining distinct. Similarly, we say that $f$ is \emph{twice strictly differentiable} at a point $x$ if
\[\Phi_2 f(u,v,w) = \dfrac{\Phi f(u,w) - \Phi f(v,w)}{u-v}\]
tends to a limit as $(u,v,w) \to (x,x,x)$ while $u$, $v$, and $w$ remaining pairwise distinct.
\end{definition}

\subsection{Partial Derivative}
Let $K/\mathbb{Q}$ be a finite Galois extension of degree $n$. Let $p \in \mathbb{Q}$ be a rational prime and $\mathfrak{p}$ be a prime ideal in $\mathcal{O}_K$ such that $\mathfrak{p} \mid p$. The discrete valuation $\ordfpp$ that extends $\ordp$ defines an ultrametric absolute value on $K$ by
\[|x|_{\ordfpp} = \sqrt[\leftroot{-5}\uproot{10}n]{|N_{K_{\ordfpp}/\mathbb{Q}_p}(x)|_{\ordp}}.\]

\begin{theorem} \label{theorem:cont_partial}
Let $K$ be a number field and $\mathfrak{p}$ a prime ideal of $\mathcal{O}_K$. The arithmetic partial derivative $\dKp$ is $\mathfrak{p}$-adically continuous on $K$.
\end{theorem}
\begin{proof}
Suppose $K/\mathbb{Q}$ is Galois. We first show that $\dKp$ is continuous at nonzero $x \in K$. Let $x_i$ be a sequence that converges to $x$ $\mathfrak{p}$-adically. Since $x \neq 0$, we can rename the sequence as $x_ix$ without loss of generality. As $i \to +\infty$, we know that
\[\ordfpp(x - x_ix) = \ordfpp(x) + \ordfpp(1-x_i) \to +\infty.\]
This implies that $\ordfpp(1-x_i) \to +\infty$ as $i \to +\infty$. As a result, we also know that $\ordfpp(x_i) = 0$ when $i \gg 0$ because if $\ordfpp(x_i) \neq 0$, then $\ordfpp(1-x_i) = \textrm{min}\{\ordfpp(1), \ordfpp(x_i)\} = 0$. Therefore $\dKp(x_i) = 0$ when $i \gg 0$. To show that $\dKp(x_i)$ converges to $\dKp(x)$ $\mathfrak{p}$-adically, it is enough to observe that 
\begin{align*}
\ordfpp(\dKp(x) - \dKp(x_ix)) &= \ordfpp(\dKp(x) - \dKp(x)x_i - \dKp(x_i)x)\\
&= \ordfpp(\dKp(x)(1-x_i)) \\
&= \ordfpp(\dKp(x)) + \ordfpp(1-x_i) \to +\infty
\end{align*}
as $i \to +\infty$. The case $x = 0$ will be covered in Theorem \ref{theorem:cont_0}.

Suppose $K/\mathbb{Q}$ is a number field, not necessarily Galois. Let $L/K$ be a finite extension such that $L/\mathbb{Q}$ is Galois. Let $\mathfrak{P}$ be a prime ideal of $\mathcal{O}_L$ such that $\mathfrak{P} \mid \mathfrak{p}$. By the previous paragraph, we know that $D_{L,\mathfrak{P}}$ is $\mathfrak{P}$-adically continuous on $L$ (and thus on $K$). Let $x_i \in K$ be a sequence that converges to $x \in K$ $\mathfrak{p}$-adically. Since $\ordfpp(y) = \nu_{\mathfrak{P}}(y)$ for all $y \in K$, we know that $x_i$ converges to $x$ $\mathfrak{P}$-adically. As $D_{L,\mathfrak{P}}$ is $\mathfrak{P}$-adically continuous on $L$, we know that $D_{L,\mathfrak{P}}(x_i)$ converges to $D_{L,\mathfrak{P}}(x)$ $\mathfrak{P}$-adically, and thus $\mathfrak{p}$-adically. This shows that $D_{L,\mathfrak{P}}$ is $\mathfrak{p}$-adically continuous on $K$. Let $T = \{\mathfrak{p}\}$. We know that by definition $\dKp(x) = D_{L,T_{L/K}}(x) = \sum_{\mathfrak{P} \mid \mathfrak{p}} D_{L,\mathfrak{P}}$. This implies that $\dKp$ is continuous on $K$.
\end{proof}

Since $\dKp$ is $\mathfrak{p}$-adically continuous on $K$, the next question is whether $\dKp$ is strictly  differentiable on $K$ with respect to the ultrametric $|\cdot|_{\ordfpp}$. 

\begin{theorem} \label{theorem:local_diff}
Let $K$ be a number field and $\mathfrak{p}$ a prime ideal of $\mathcal{O}_K$. The arithmetic partial derivative $\dKp$ is strictly differentiable and twice strictly differentiable (with respect to the ultrametric $|\cdot|_{\ordfpp}$) at every nonzero $x \in K$.
\end{theorem}
\begin{proof}
Let $L/K$ be a finite extension such that $L/\mathbb{Q}$ is Galois. Let $T = \{\mathfrak{p}\}$. We have $\dKp(x) = D_{L,T_{L/K}}(x) = \sum_{\mathfrak{P} \mid \mathfrak{p}} D_{L,\mathfrak{P}}$.

We first show that $\dKp$ is strictly differentiable at $x \neq 0$. Suppose a sequence $(u_i, v_i)$ converges to $(x, x)$ $\mathfrak{p}$-adically while $u_i$ and $v_i$ remaining distinct. This implies that $(u_i,v_i)$ converges to $(x,x)$ $\mathfrak{P}$-adically. When $i \gg 0$, we have $\nu_{\mathfrak{P}}(u_i) = \nu_{\mathfrak{P}}(v_i) = \nu_{\mathfrak{P}}(x)$.
We can compute
\begin{align*}
\Phi \dKp(u_i,v_i) &= \dfrac{\dKp(u_i)-\dKp(v_i)}{u_i-v_i} = \dfrac{\sum_{\mathfrak{P} \mid \mathfrak{p}} D_{L, \mathfrak{P}}(u_i) - \sum_{\mathfrak{P} \mid \mathfrak{p}} D_{L, \mathfrak{P}}(v_i)}{u_i-v_i} \\
&= \dfrac{ \sum_{\mathfrak{P} \mid \mathfrak{p}} \frac{u_i\nu_{\mathfrak{P}}(x)}{pg(p,L)} - \sum_{\mathfrak{P} \mid \mathfrak{p}} \frac{v_i\nu_{\mathfrak{P}}(x)}{pg(p,L)}}{u_i-v_i} = \sum_{\mathfrak{P} \mid \mathfrak{p}} \frac{\nu_{\mathfrak{P}}(x)}{pg(p,L)} = \dfrac{\dKp(x)}{x}.
\end{align*}
Therefore the limit of $\Phi \dKp(u_i, v_i)$ is equal to $\dKp(x)/x$ as $i \to +\infty$. This shows that $\dKp$ is strictly differentiable at any nonzero $x \in K$, and the derivative of $\dKp$ is a constant function, defined by
\[(\dKp)'(x) = \dKp(x)/x = \ldKp(x).\]
We then show that $\dKp$ is twice strictly differentiable at nonzero points. Suppose a sequence $(u_i, v_i, w_i)$ converges to $(x,x,x)$ ${\mathfrak{p}}$-adically while $u_i$, $v_i$, and $w_i$ remaining pairwise distinct. Then for all $i \gg 0$, we have
\[\Phi_2 \dKp(u_i,v_i,w_i) = \dfrac{\Phi \dKp(u_i,w_i) - \Phi \dKp(v_i, w_i)}{u_i - v_i} = \dfrac{0}{u_i-v_i} = 0.\]
Hence $\dKp$ is twice strictly differentiable at nonzero points and the second derivative is the constant zero function.
\end{proof}

\begin{theorem} \label{theorem:local_not_diff_0}
Let $K$ be a number field and $\mathfrak{p}$ a prime ideal of $\mathcal{O}_K$. The arithmetic partial derivative $\dKp$ is not strictly differentiable (with respect to the ultrametric $|\cdot|_{\ordfpp}$) at $0$.
\end{theorem}
\begin{proof}
This theorem is a direct corollary of a more generalized Theorem \ref{theorem:not_diff}.
\end{proof}

\begin{remark}
Theorems \ref{theorem:cont_partial}, \ref{theorem:local_diff}, and \ref{theorem:local_not_diff_0} hold in the local case of finite extensions over $\mathbb{Q}_p$.
\end{remark}

\subsection{Subderivative}

\begin{theorem} \label{theorem:cont_0}
Let $K/\mathbb{Q}$ be a number field and $\mathfrak{p}$ be a prime ideal of $\mathcal{O}_K$. Let $T$ be a nonempty set of prime ideals in $\mathcal{O}_K$. The arithmetic subderivative $D_{K,T}$ is $\mathfrak{p}$-adically continuous at $x = 0$.
\end{theorem}
\begin{proof}
Let $L/K$ be a finite extension such that $L/\mathbb{Q}$ is Galois. Suppose $x_i \in K$ is a sequence that converges to $x$ $\mathfrak{p}$-adically in $K$. Let $\mathfrak{P}$ be a prime ideal of $\mathcal{O}_L$ such that $\mathfrak{P} \mid \mathfrak{p}$. Then $x_i$ converges to $x$ $\mathfrak{P}$-adically in $L$. Hence
\[\lim_{i \to +\infty} \nu_{\mathfrak{P}}(x - x_i) = \lim_{i \to +\infty} \nu_{\mathfrak{P}}(x_i) = +\infty.\]
We have
\begin{align*}
  \ordfpp(D_{K,T}(x_i)) &= \nu_{\mathfrak{P}}(D_{L,T_{L/K}}(x_i)) \\
  &= \nu_{\mathfrak{P}}\Big(x_i \sum_{\mathfrak{Q} \in T_{L/K}, \mathfrak{Q} \mid q} \dfrac{\nu_{\mathfrak{Q}}(x_i)}{qg(q,L)}  \Big)\\
  &= \nu_{\mathfrak{P}}(x_i) + \nu_{\mathfrak{P}}\Big( \sum_{\mathfrak{Q} \in T_{L/K}, \mathfrak{Q} \mid q} \dfrac{\nu_{\mathfrak{Q}}(x_i)}{qg(q,L)}  \Big) \\  
  &= \nu_{\mathfrak{P}}(x_i) + \nu_{\mathfrak{P}}\Big(\dfrac{1}{[L:\mathbb{Q}]}\sum_{\mathfrak{Q} \in T_{L/K}, \mathfrak{Q} \mid q} \dfrac{\nu_{\mathfrak{Q}}(x_i)e(q,L)f(q,L)}{q}  \Big) \\
  &\geq \nu_{\mathfrak{P}}(x_i) - \nu_{\mathfrak{P}}([L:\mathbb{Q}])) - \nu_{\mathfrak{P}}(\prod_{\mathfrak{Q} \in T_{L/K}, \mathfrak{Q} \mid q}q).
\end{align*}
As $\displaystyle \lim_{i \to +\infty} \nu_{\mathfrak{P}}(x_i) = +\infty$, we have 
\[\displaystyle \lim_{i \to _+\infty} \ordfpp(D_{K,T}(x) - D_{K,T}(x_i)) = +\infty. \qedhere\]
\end{proof}

\begin{corollary} \label{corollary:generalization1}
Let $T$ be a nonempty set of (rational) prime numbers. The arithmetic subderivative $D_{\mathbb{Q},T}$ is $p$-adically continuous at $x = 0$.
\end{corollary}

\begin{theorem} \label{theorem:not_diff}
Let $K/\mathbb{Q}$ be a number field and $\mathfrak{p}$  a prime ideal of $\mathcal{O}_K$. Let $T$ be a nonempty set of prime ideals in $\mathcal{O}_K$. The arithmetic subderivative $D_{K,T}: K \to K$ is not strictly differentiable (with respect to the ultrametric $|\cdot|_{\ordfpp}$) at $0$.
\end{theorem}
\begin{proof}
Let $L/K$ be a finite extension such that $L/\mathbb{Q}$ is Galois. Let $\mathfrak{P}$ be a prime ideal of $\mathcal{O}_L$ such that $\mathfrak{P} \mid \mathfrak{p} \mid p$. 

We prove this theorem in two cases. First, we assume that there exists a prime ideal $\mathfrak{p}' \in T$ such that $\mathfrak{p}' \mid p$. Let $m_p$ be the number of prime ideals in $T_{L/K}$ that divide $p$. For positive integer $i \geq 1$, define $u_i = p^{i+1}, v_i = p^{i}$. It is clear that $u_i \neq v_i$ and $(u_i, v_i)$ converges to $(0,0)$ $\mathfrak{p}$-adically. We can compute the difference quotient
\begin{align*}
\Phi D_{K,T} (u_i, v_i) &= \dfrac{D_{K,T}(u_i) - D_{K,T}(v_i)}{u_i - v_i} = \dfrac{D_{L,T_{L/K}}(u_i) - D_{L,T_{L/K}}(v_i)}{u_i-v_i} \\
&= \dfrac{\frac{(i+1)p^{i+1}m_p}{pg(p,L)} - \frac{ip^{i}m_p}{pg(p,L)}}{p^{i+1}-p^i} = \dfrac{m_p}{g(p,L)}\dfrac{(i+1)p-i}{p^2-p}.
\end{align*}
The $\mathfrak{p}$-adic valuation of $\Phi D_{K,T} (u_i, v_i)$ is greater than or equal to $\ordfpp(m_p) - \ordfpp(g(p,L))$ if $p \mid i$ and is equal to $\ordfpp(m_p) - \ordfpp(g(p,L)) -1 $ if $p \mid i$. Hence $\Phi D_{K,T} (u_i, v_i)$ does not have a limit as the sequence $(u_i,v_i) \to (0,0)$.

Second, we assume that there does not exist a prime ideal $\mathfrak{p}' \in T$ such that $\mathfrak{p}' \mid p$. Let $\mathfrak{q} \in T$ be such that $\mathfrak{q} \nmid p$ and $\mathfrak{Q} \in T_{L/K}$ such that $\mathfrak{Q} \mid \mathfrak{q} \mid q$. Let $m_q$ be the number of prime ideals in $T_{L/K}$ that divide $q$. For positive integer $i \geq 1$, define $u_i = (pq)^{i+1}, v_i = (pq)^i$. It is clear that $u_i \neq v_i$ and $(u_i, v_i)$ converges to $(0,0)$ $\mathfrak{p}$-adically. We can compute the difference quotient
\begin{align*}
\Phi D_{K,T} (u_i, v_i) &= \dfrac{D_{K,T}(u_i) - D_{K,T}(v_i)}{u_i - v_i} = \dfrac{D_{L,T_{L/K}}(u_i) - D_{L,T_{L/K}}(v_i)}{u_i-v_i} \\
&= \dfrac{\frac{(i+1)(pq)^{i+1}m_q}{qg(q,K)} - \frac{i(pq)^im_q}{qg(q,K)}}{(pq)^{i+1}-(pq)^{i}} = \frac{m_q}{g(q,K)} \dfrac{(i+1)pq-i}{pq^2-q}.
\end{align*}
The $\mathfrak{p}$-adic valuation of $\Phi D_{K,T} (u_i, v_i)$ is greater than or equal to $\ordfpp(m_q) - \ordfpp(g(q,K)) + 1$ if $p \mid i$ and is equal to $\ordfpp(m_q) - \ordfpp(g(q,K))$ if $p \mid i$. Hence $\Phi D_{K,T} (u_i, v_i)$ does not have a limit as the sequence $(x_i,y_i) \to (0,0)$.
\end{proof}

\begin{theorem} \label{theorem:discontinuous}
Let $K/\mathbb{Q}$ be a number field of degree $n$. Let $\mathfrak{p}$ be a prime ideal of $\mathcal{O}_K$ with $\mathfrak{p} \mid p$. Let $\{\mathfrak{p}\} \neq T$ be a nonempty set of prime ideals in $\mathcal{O}_K$ such that there exists a prime ideal in $T$ that does not divide $p$. Then the arithmetic subderivative $D_{K,T}: K \to K$ is $\mathfrak{p}$-adically discontinuous at every nonzero $x \in K$.
\end{theorem}
\begin{proof}
We first assume $K/\mathbb{Q}$ is Galois. For each prime $q \in \mathbb{P}$, let $r_q$ be the number of prime ideals $\mathfrak{q} \in T$ such that $\mathfrak{q} \mid q$. Let $\mathbb{P}_T := \{q \in \mathbb{P} \mid r_q \neq 0, q \neq p\}$ and we know $0 \leq \ordp(g(q,K)) \leq \ordp(n)$ for all $q \in \mathbb{P}_T$. Let $q_0 \in \mathbb{P}_T$ be a prime such that $\ordp(g(q_0,K)) = \mathrm{min}\{\ordp(g(q,K)) \mid q \in \mathbb{P}_T\}$. Let $M := \textrm{max}\{\ordp(j): 1 \leq j \leq n \}+1$. For each integer $i \geq 1$, the Dirichlet's theorem on arithmetic progression implies there are infinitely many primes in the arithmetic progression $q_0^{p^M}, q_0^{p^M} + p^i, q_0^{p^M} + 2p^i, \ldots$. Set $n_0:=0$. For each $i \geq 1$, let $n_i > n_{i-1}$ be a positive integer such that $q_i := q_0^{p^M} + n_ip^i$ is a prime, that is, one prime from each arithmetic progression. Hence we know that $p, q_0, q_1, q_2, \dots$ is a list of pairwise distinct prime numbers. Let $x_i := q_0^{p^M}x/q_i \in K$. One can show that 
\[\lim_{i \to +\infty} \ordfpp(x - x_i) = \lim_{i \to +\infty} \ordfpp\Big(\dfrac{x n_ip^i}{q_i}\Big) = \lim_{i \to +\infty} \ordfpp(x n_ip^i) = +\infty.\]
This means that the sequence $x_i$ converges to $x$ $\mathfrak{p}$-adically. We now show that $D_{K,T}(x_i)$ does not converge to $D_{K,T}(x)$ $\mathfrak{p}$-adically. We have
\begin{align*}
    D_{K,T}(x) - D_{K,T}(x_i) &= D_{K,T}(x) - \Big(\dfrac{q_0^{p^M}}{q_i} D_{K,T}(x) + x D_{K,T}\big(\dfrac{q_0^{p^M}}{q_i}\big)\Big) \\
    &= \dfrac{n_ip^i}{q_i} D_{K,T}(x) - x \dfrac{D_{K,T}(q_0^{p^M}) q_i - q_0^{p^M} D_{K,T}(q_i)}{q_i^2} \\
   &= \dfrac{n_ip^i}{q_i} D_{K,T}(x) - \dfrac{x r_{q_0}p^M q_0^{p^M-1}}{g(q_0,K)q_i} + \dfrac{x q_0^{p^M} D_{K,T}(q_i) }{q_i^2}. \\
\end{align*}

We analyze the $\mathfrak{p}$-adic valuation of each of three summands separately. For the first summand, we have
\[\lim_{i \to +\infty} \ordfpp\Big(\dfrac{n_ip^i}{q_i} D_{K,T}(x)\Big) = \lim_{i \to +\infty} \ordfpp(p^i) = +\infty.\]
For the second summand, as $i \gg 0$, we have
\[\ordfpp\Big(\dfrac{x r_{q_0} p^M q_0^{p^M-1}}{g(q_0,K)q_i}\Big) = \ordfpp\Big(\dfrac{x r_{q_0} p^M}{g(q_0,K)}\Big) = \ordfpp\Big(\dfrac{xr_{q_0}}{g(q_0,K)}\Big) +  M .\]
For the third summand, if $q_i \notin \mathbb{P}_T$, then $D_{K,T}(q_i) = 0$ so it has no contribution to the $\mathfrak{p}$-adic valuation. On the other hand, if $q_i \in \mathbb{P}_T$, then we have
\[\ordfpp\Big(\dfrac{x q_0^{p^M} D_{K,T}(q_i) }{q_i^2} \Big) = \ordfpp\Big(\dfrac{x q_0^{p^M} r_{q_i}}{g(q_i,K)q_i^2}\Big) = \ordfpp\Big(\dfrac{x r_{q_i}}{g(q_i,K)}\Big).\]
Since $1 \leq r_{q_i} \leq n$, we know that $M > \ordp(r_{q_i})$ by definition. We also know that $\ordp(g(q_0,K)) \leq \ordp(g(q_i,K))$ for all $i \geq 1$. Hence
\[\ordfpp\Big(\dfrac{x r_{q_0}}{g(q_0,K)}\Big) +M > \ordfpp\Big(\dfrac{x r_{q_i}}{g(q_i,K)}\Big).\]
This implies that 
\begin{equation*}
\ordfpp(D_{K,T}(x) - D_{K,T}(x_i)) = \begin{cases}
\ordfpp\Big(\frac{x r_{q_i}}{g(q_i,K)}\Big), & \text{if $q_i \in \mathbb{P}_T$};\\
\ordfpp\Big(\frac{x r_{q_0}}{g(q_0,K)}\Big) +M, & \text{if $q_i \notin \mathbb{P}_T$}.
\end{cases}
\end{equation*}
This implies that
\[\lim_{i \to +\infty} \ordfpp(D_{K,T}(x) - D_{K,T}(x_i)) \neq +\infty.\]

Now we assume that $K/\mathbb{Q}$ is not necessarily Galois. Let $L/K$ be a finite extension such that $L/\mathbb{Q}$ is Galois, and $\mathfrak{P}$ a prime ideal of $\mathcal{O}_L$ such that $\mathfrak{P} \mid \mathfrak{p}$. Since $T$ contains a prime ideal that does not divide $p$, we know that $T_{L/K}$ also contains a prime ideal that does not divide $p$. Let $x_i \in K$ be defined as above. Then we know that $x_i$ converges to $x$ $\mathfrak{p}$-adically in $K$, and thus $\mathfrak{P}$-adically in $L$ since $\ordfpp$ and $\nu_{\mathfrak{P}}$ agree on $K$. Since $L/\mathbb{Q}$ is Galois, we know that
\[\lim_{i \to +\infty}(\nu_{\mathfrak{P}}(D_{L,T_{L/K}}(x_i) - D_{L,T_{L/K}}(x))) \neq +\infty.\] 
Hence 
\[\lim_{i \to +\infty}(\ordfpp(D_{K,T}(x_i) - D_{K,T}(x)) = \lim_{i \to +\infty} (\nu_{\mathfrak{P}}(D_{L,T_{L/K}}(x_i) - D_{L,T_{L/K}}(x))) \neq +\infty.\]
This shows that $D_{K,T}$ is discontinuous at $x$. 
\end{proof}

\begin{corollary} \label{corollary:generalization2}
Let $\{p\} \neq T$ be a nonempty set of prime numbers. The arithmetic subderivative $D_{\mathbb{Q}, T}$ is $p$-adically discontinuous at any nonzero $x \in \mathbb{Q}$.
\end{corollary}
\begin{proof}
Apply Theorem \ref{theorem:discontinuous} by taking $K = \mathbb{Q}$ and $\mathfrak{p} = (p)$. 
\end{proof}

\begin{remark}
Corollaries \ref{corollary:generalization1} and \ref{corollary:generalization2} together give answers to all open questions about $p$-adic continuity and discontinuity of arithmetic subderivative over $\mathbb{Q}$ listed in \cite[Section 7]{HMT3}.
\end{remark}

The only case that is left for consideration is when all prime ideals in $T$ sit above the same $p$. This case will be fully answered by the next theorem when we assume $T$ is finite.

\begin{theorem} \label{theorem:discontinuous_special}
Let $K/\mathbb{Q}$ be a number field of degree $n$. Let $\mathfrak{p}$ be a prime ideal of $\mathcal{O}_K$ with $\mathfrak{p} \mid p$. Let $\{\mathfrak{p}\} \neq T$ be a nonempty finite set of prime ideals in $\mathcal{O}_K$. Then the arithmetic subderivative $D_{K,T}: K \to K$ is $\mathfrak{p}$-adically discontinuous at any nonzero $x \in K$.
\end{theorem}
\begin{proof}
We first assume $K/\mathbb{Q}$ is Galois. Let $T \, \backslash \, \{\mathfrak{p}\} = \{\mathfrak{p}_1, \ldots, \mathfrak{p}_n\}$. By the Chinese remainder theorem, for each $i \geq 1$, there exists $x_i \in K$ such that $\ordfpp(1 - x_i) = i$, $\nu_{\mathfrak{p}_1}(x_i) = 1$, and $\nu_{\mathfrak{p}_j}(x_i) =0$ for $2 \leq j \leq n$. This implies that $\ordfpp(x_i) = 0$. Hence for all $i \geq 1$, we have \[D_{K,T}(x_i) = \dfrac{x_i}{p_1g(p_1,K)}.\]

The sequence $x_ix$ converges to $x$ $\mathfrak{p}$-adically because as $i \to +\infty$, we have
\[\ordfpp(x - x_ix) = \ordfpp(1-x_i)+\ordfpp(x) \to +\infty.\]
On the other hand, $D_{K,T}(x_ix)$ does not converge to $D_{K,T}(x)$ $\mathfrak{p}$-adically because as $i \gg 0$, we have
\begin{align*}
    \ordfpp(D_{K,T}(x) - D_{K,T}(x_ix)) &= \ordfpp(D_{K,T}(x) - x_iD_{K,T}(x) - x D_{K,T}(x_i)) \\
    &= \ordfpp\Big(D_{K,T}(x)(1-x_i) - \dfrac{x x_i}{p_1g(p_1,K)}\Big) \\
    &= \ordfpp(x) - \ordfpp(p_1) - \ordfpp(g(p_1,K)).
\end{align*}
Hence
\[\lim_{i \to +\infty} \ordfpp(D_{K,T}(x) - D_{K,T}(x_ix)) \neq +\infty,\]
and $D_{K,T}$ is discontinuous at $x$.

If $K/\mathbb{Q}$ is not necessarily Galois, then one can prove that $D_{K,T}$ is discontinuous at $x$ using the same strategy as in Theorem \ref{theorem:discontinuous}.
\end{proof}

\providecommand{\bysame}{\leavevmode\hbox to3em{\hrulefill}\thinspace}

\end{document}